\documentclass[reqno]{amsart}
\usepackage{appendix}
\usepackage{prftree}
\usepackage{mathtools}
\usepackage{amsthm}
\usepackage{amssymb}
\usepackage{wasysym}
\usepackage{cleveref}
\usepackage[numbers]{natbib}
\usepackage{tikz}
\usepackage{enumitem}

\author{Nick Bezhanishvili and Antonio Maria Cleani}
\title[Pre-filtrations, pre-stable canonical rules and the KM-isomorphism]{Pre-filtrations, pre-stable canonical rules and the Kuznetsov-Muravitsky isomorphism}

\usetikzlibrary{positioning}
\usetikzlibrary{calc}
\usetikzlibrary{arrows.meta,arrows}
\tikzset{
	modal/.style={>=stealth,shorten >=1pt,shorten <=1pt,auto,node distance=1.5cm,
		semithick},
	world/.style={rectangle,draw,minimum size=0.8cm,fill=gray!15},
	point/.style={circle,draw,inner sep=0.5mm,fill=black},
	reflexive above/.style={->,loop,looseness=5,in=120,out=60},
	reflexive below/.style={->,loop,looseness=5,in=240,out=300},
	reflexive left/.style={->,loop,looseness=5,in=150,out=210},
	reflexive right/.style={->,loop,looseness=5,in=30,out=330}
}

\tikzset{roundnode/.style={draw,circle,inner sep=2pt,outer sep=0pt}
}
\tikzset{irr/.style={draw,circle,fill,inner sep=2pt,outer sep=0pt}
}

\newcommand{\logic}[1]{\mathtt{#1}}

\renewcommand{\int}[2][g]{\llbracket #2\rrbracket^{#1}}

\newcommand{\dualspa}[1]{\alg{#1}_*}
\newcommand{\dualalg}[1]{\spa{#1}^*}

\newcommand{\dualkr}[1]{\alg{#1}_+}

\newcommand{\clop}[2][]{\op{Clop}_{#1}(\spa{#2})}
\newcommand{\clopup}[2][]{\op{ClopUp}_{#1}(\spa{#2})}
\newcommand{\clopdown}[2][]{\op{ClopDown}_{#1}(\spa{#2})}

\newcommand{\upset}[2][]{{\Uparrow_{#1}} #2}
\newcommand{\downset}[2][]{{\Downarrow_{#1}} #2}
\renewcommand{\max}[2][]{{\mathit{max}_{#1}} (#2)}

\newcommand{\pas}[2][]{\mathit{pas}_{#1}({#2})}

\renewcommand{\sim}{\mathit{sim}}
\newcommand{\clm}{\mathit{clm}}
\newcommand{\sep}{\, |\,}

\newcommand{\alg}[1]{\mathfrak{#1}}
\newcommand{\spa}[1]{\mathfrak{#1}}
\newcommand{\class}[1]{\mathcal{#1}}
\newcommand{\op}[1]{\mathsf{#1}}
\newcommand{\lat}[1]{\mathbf{#1}}

\newcommand{\scrmod}[2]{\mu(\alg{#1}, {#2})}

\newcommand{\scr}[2]{{\xi}(\alg{#1}, {#2})}

\newcommand{\ruletrans}[2]{\mu_\circ(\alg{#1}, {#2})}

\newcommand{\scrsi}[2]{\eta(\alg{#1}, {#2})}

\newcommand{\least}[1][]{\tau_{#1}}
\newcommand{\greatest}[1][]{\sigma_{#1}}
\newcommand{\fragment}[1][]{\rho_{#1}}

\theoremstyle{plain}
\newtheorem{theorem}{Theorem}[section]
\newtheorem{proposition}[theorem]{Proposition}
\newtheorem{lemma}[theorem]{Lemma}
\newtheorem{corollary}[theorem]{Corollary}

\theoremstyle{definition}
\newtheorem{definition}[theorem]{Definition}

\newtheorem{remark}[theorem]{Remark}

\AfterEndEnvironment{definition}{\noindent\ignorespaces}
\AfterEndEnvironment{proposition}{\noindent\ignorespaces}
\AfterEndEnvironment{theorem}{\noindent\ignorespaces}
\AfterEndEnvironment{lemma}{\noindent\ignorespaces}
\AfterEndEnvironment{example}{\noindent\ignorespaces}
\AfterEndEnvironment{remark}{\noindent\ignorespaces}
\AfterEndEnvironment{claim}{\noindent\ignorespaces}
\AfterEndEnvironment{corollary}{\noindent\ignorespaces}

\date{\today}

\begin{document}

\maketitle

\begin{abstract}
We introduce pre-filtrations and pre-stable canonical rules for the Kuznetsov–Muravitsky system of intuitionistic modal logic and provide a new proof of the Kuznetsov–Muravitsky isomorphism, along with several preservation results. The proofs employ these rules and a duality between modal (Heyting) algebras and their corresponding order-topological spaces.
\end{abstract}

\section{Introduction}

The Kuznetsov–Muravitsky logic 
$\logic{KM}$
is among the most curious and intriguing  systems of Intuitionistic Modal Logic. It was introduced by Kuznetsov and Muravitsky in \cite{kuznetsov1980provability, Kuznetsov1978PIL}. One of the main motivations for studying this system was its connection to the modal provability logic 
$\logic{GL}$. According to Kuznetsov (see \cite{Muravitsky2014LKaB}), understanding a logical system amounts to understanding the behavior of this system and its ``neighbors''—that is, the extensions of this logic. From this point of view, the true intuitionistic counterpart of the system  
$\logic{GL}$ is the system whose lattice of normal extensions is isomorphic to the lattice of normal extensiosn of 
$\logic{GL}$. The Kuznetsov–Muravitsky system precisely satisfies this condition. It was proved in \cite{KuznetsovMuravitsky1986OSLAFoPLE, muravitsky1985correspondence, muravitskii1989correspondence} that the lattice of normal extensions of 
$\logic{GL}$ is isomorphic to the lattice of normal extensions of 
$\logic{KM}$. Following Esakia \cite{Esakia2006TMHCaCMEotIL}, we will refer to this result as the Kuznetsov–Muravitsky isomorphism. 

We refer to \cite{Muravitsky2014LKaB} and \cite{Litak2014CMwPS} for a comprehensive history and overview of the system $\logic{KM}$. Here, we would like to highlight one additional illuminating  result concerning $\logic{KM}$, commonly known as Kuznetsov's theorem: adding the $\logic{KM}$ axioms to any superintuitionistic logic yields a conservative extension of that logic. In algebraic terms, every variety of Heyting algebras is generated by the Heyting reducts of $\logic{KM}$-algebras. This result was first announced by Kuznetsov in \cite{kuznetsov1985proof}. Kuznetsov’s original proof was proof-theoretic in nature and a bit  sketchy. It was later reproduced and expanded by Muravitsky \cite{muravitsky2017equipollence}. More recently, Jibladze and E.~Kuznetsov \cite{JK25} provided a purely algebraic proof of this result.

In this paper,\footnote{This paper is partially based on the second-named author's Master's thesis \cite{Cleani2021TEVSCR}.} we concentrate on the Kuznetsov–Muravitsky isomorphism, on a related isomorphism between the lattice of normal extensions of the modalized Heyting Calculus and  of the modal logic $\logic{K4.Grz}$, and the corresponding preservation results. Our contribution consists in giving novel proofs of these known results. 

Recently, a new proof of the lattice isomorphism between superintuitionistic logics and normal extensions of the modal logic $\logic{Grz}$---known as the Blok-Esakia ismorphism---was given by the authors of this paper in \cite{BezhanishviliBETVSCR}. The method is based on the  stable canonical formulas and rules \cite{BezhanishviliEtAl2016SCR, BezhanishviliEtAl2016CSL} and makes use of the duality between modal, Heyting, and frontal Heyting algebras on the one hand, and the corresponding order-topological structures on the other. Using this technique, a new dual proof of the Blok–Esakia isomorphism between superintuitionistic logics and the normal extensions of the modal logic $\logic{Grz}$ was obtained in \cite{BezhanishviliBETVSCR}. The other existing proofs are due to Blok \cite{Blok1976VoIA} (see also \cite{Stronkowski2018OtBETfUC}), Esakia (which has never been published), and  Zakharyaschev (see, e.g., \cite[Section 9.6]{ChagrovZakharyaschev1997ML}).  Our approach was also generalized to a new proof of a lattice  isomorphism between bi-superintuitionistic logics and of normal extensions of tense-$\logic{Grz}$ (originally proved by Wolter \cite{Wolter1998OLwC}). Furthermore, the technique was also extended to  a new proof of the Kuznetsov–Muravitsky isomorphism \cite{BezhanishviliBETVSCR}. Note that for this, one only needs a theory of stable canonical rules over the modal system $\logic{K4}$.

In the present paper we revise this method by introducing the apparatus of \emph{pre-stable canonical  rules} for $\logic{KM}$ and $\logic{GL}$. This refinement allows us to obtain new proofs of the preservation results (which was not achieved in \cite{BezhanishviliBETVSCR}), since it requires a theory of algebra-based rules for both $\logic{KM}$ and $\logic{GL}$. To develop these, we revisit constructions related to filtration for $\logic{KM}$ and $\logic{GL}$ used in \cite{Castiglioni2012TFMPftVoHAwS} and \cite{Boolos1993TLoP}, and related, in the $\logic{KM}$ case, to a construction of Muravitsky \cite{Muravitskii1981FAotICatEoaEHNM}. We call these constructions \emph{pre-filtrations}, and the resulting rules \emph{pre-stable canonical rules}. Equipped with this technique, we revisit the proof strategy of \cite{BezhanishviliBETVSCR} to obtain an alternative proof of the Kuznetsov–Muravitsky isomorphism and of an isomorphism between the lattice of normal extensions of the modalized Heyting Calculus and that of the modal system $\logic{K4.Grz}$. The latter result was first announced by Esakia \cite{Esakia2006TMHCaCMEotIL} and later proved by Litak \cite{Litak2014CMwPS} and Muravitsky \cite{Murav17LU} using different  approaches. We will refer to it as an Esakia theorem.  We then prove some preservation results concerning the  Kuznetsov–Muravitsky isomorphism, extending proof strategies from \cite{BezhanishviliBETVSCR} that take full advantage of  pre-stable canonical rules for $\logic{KM}$.

With this paper we hope to provide a new perspective on the Kuznetsov–Mura\-vitsky system $\logic{KM}$. We are honored and extremely happy to be able to contribute this work to the volume dedicated to Professor Kuznetsov's memory. The impact of Kuznetsov's work on many areas of mathematical logic is difficult to overstate. We hope that the technique of pre-stable canonical  rules for $\logic{KM}$, together with the isomorphism and preservation proofs, will serve as another stepping stone toward understanding this beautiful and deeply intriguing system of intuitionistic modal logic.

\section{Preliminaries}


\subsection{Rule Systems} 

	Throughout the paper we fix a countably infinite set of propositional variables $\mathit{Prop}$. The set $\mathit{Frm}_\sim$ of \emph{superintuitionistic modal formulae} is defined recursively as follows:
    \[\varphi::= p \sep  \bot \sep \top \sep \varphi \land \varphi \sep \varphi\lor \varphi \sep \varphi\to \varphi \sep {\boxtimes\varphi}, \]
    where $p\in \mathit{Prop}$. The set of $\mathit{Frm}_\clm$ \emph{classical modal formulae} is defined recursively as follows: 
    \[\varphi::= p \sep  \bot \sep \top \sep \varphi \land \varphi \sep \neg \varphi \sep   \square\varphi, \]
    where, again, $p\in \mathit{Prop}$. We abbreviate the classical connectives $\to$ and $\lor$ in the usual way. 

    For $\heartsuit\in \{\sim, \clm\}$, the set $\mathit{Rul}_\heartsuit$ of \emph{$\heartsuit$-rules} is defined as the set of all ordered pairs $(\Gamma, \Delta)$ such that $\Gamma, \Delta$ are finite subsets of $\mathit{Frm}_\heartsuit$. We adopt the convention of writing $\Gamma/\Delta$ for the rule $(\Gamma, \Delta)$. Intuitively, a rule $\Gamma/\Delta$ says that if all the formulae in $\Gamma$ hold, then at least one of the formulae in $\Delta$ holds.

    For $\heartsuit\in \{\sim, \clm\}$, a \emph{$\heartsuit$-rule system} is defined as a set $\logic{S}$ of $\heartsuit$-rules satisfying the following conditions:
        \begin{enumerate}
			\item If $\Gamma/\Delta\in \logic{S}$, then $ s[\Gamma]/ s[\Delta]\in \logic{S}$ for all substitutions $s$ (structurality);
			\item $\varphi/\varphi\in \logic{S}$ for every $\heartsuit$-formula $\varphi$ (reflexivity);
			\item If $\Gamma/\Delta\in \logic{S}$, then $\Gamma;\Gamma'/\Delta;\Delta'\in \logic{S}$ for any finite sets of $\heartsuit$-formulae $\Gamma',\Delta'$ (monotonicity);
			\item If $\Gamma/\Delta;\varphi\in\logic{S}$ and $\Gamma;\varphi/\Delta\in \logic{S}$, then $\Gamma/\Delta\in \logic{S}$ (cut);
			\item $\varphi, \varphi\to \psi/\psi\in \logic{S}$ for every $\heartsuit$-formulae $\varphi, \psi$ (modus ponens);
			\item ${\varphi/{\odot \varphi}}$ for every $\heartsuit$-formula $\varphi$, where ${\odot=\boxtimes}$ if ${\heartsuit=\sim}$ and ${\odot=\square}$ if ${\heartsuit=\clm}$ (necessitation).
        \end{enumerate}
    Here by a \emph{substitution} we mean a function $s:\mathit{Prop}\to \mathit{Frm}_\heartsuit$ which commutes with all the primitive connectives of $\mathit{Frm}_\heartsuit$. 

    If $\class{S}$ is a set of $\heartsuit$-rule systems and  $\Sigma, \Xi$ are sets of $\heartsuit$-rules, we write $\Xi\oplus_{\class{S}}\Sigma$ for the least $\heartsuit$-rule system in $\class{S}$, if it exists, extending both $\Xi$ and $\Sigma$. Likewise, we write $\Xi\oplus_{\class{S}}\Sigma$ for the greatest $\heartsuit$-rule system in $\class{S}$, if it exists, contained in both $\Xi$ and $\Sigma$. When $\class{S}$ is clear from context, we write simply $\oplus$ and $\otimes$ instead of $\oplus_{\class{S}}$ and $\otimes_{\class{S}}$. If $\logic{S}$ is a $\heartsuit$-rule system, we write $\lat{NExt}(\logic{S})$ for the set of all $\heartsuit$-rule systems extending $\logic{S}$.
    
    The set of all $\heartsuit$-rule systems forms a complete lattice under the operations $\oplus$ and $\otimes$; the meet $\otimes$ in fact coincides with intersection. Given two $\heartsuit$-rule systems $\logic{S,S'}$ and a set of $\heartsuit$-rules $\Xi$, we say that $\Xi$ \emph{axiomatizes $\logic{S}$ over $\logic{S'}$} when $\logic{S'}\oplus \Xi=\logic{S}$, where the join is that of the lattice of all $\heartsuit$-rule systems. We say simply that $\Xi$ \emph{axiomatizes $\logic{S}$} when $\Xi$ {axiomatizes $\logic{S}$} over $\logic{S'}$ and $\logic{S'}$ is the least $\heartsuit$-rule system. 
    
    A \emph{$\heartsuit$-logic} is a $\heartsuit$-rule system axiomatized by a set of assumption free, single-conclusion rules---that is, by a set of rules of the form $/\varphi$ for some $\heartsuit$-formula $\varphi$. Logics in this sense correspond one-to-one with logics conceived of as sets of formulae closed under appropriate conditions, which is the most common conception of logics in the literature. Indeed, when $\logic{L}$ is a logic in the latter sense, the rule system (in the appropriate signature) axiomatized by all rules $/\varphi$ for $\varphi\in \logic{L}$ is a logic in our sense.

     If $\class{S}$ is a set of $\heartsuit$-rule systems and  $\logic{L, L'}$ are $\heartsuit$-logics, we write $\logic{L}\sqcup_\class{S}\logic{L'}$ and $\logic{L}\sqcap_\class{S}\logic{L'}$, respectively, for the least $\heartsuit$-\emph{logic} in $\class{S}$ extending both $\logic{L}$ and $\logic{L'}$ and for the greatest $\heartsuit$-\emph{logic} in $\class{S}$ contained in both $\logic{L}$ and $\logic{L'}$, provided they exist. As before, we omit subscripts when they can be inferred from context.  It turns out that $\sqcup_\class{S}$ and $\sqcap_\class{S}$ are always well defined when $\class{S}$ is the set of all $\heartsuit$-rule systems, as are their infinitary versions. Thus the set of all $\heartsuit$-logics carries a complete lattice. If $\logic{L}$ is a $\heartsuit$-logic, we write $\lat{NExtL}(\logic{L})$ for the lattice of all $\heartsuit$-logics extending $\logic{L}$. 

     It is easy to see that, over the set of all $\heartsuit$-rule systems, the join $\logic{L}\oplus\logic{L'}$ of two $\heartsuit$-logics is always a $\heartsuit$-logic. Thus $\sqcup$ coincides with $\oplus$. However, $\logic{L}\otimes\logic{L'}$ may not be a logic. It will be useful to give a characterization of $\sqcap$ in terms of $\otimes$. If $\logic{S}$ is a $\heartsuit$-rule system, let $\op{Taut}(\logic{S})$ be the logic axiomatized by all rules of the form $/\varphi$ such that $/\varphi\in \logic{S}$. 
     \begin{proposition}
     The identity 
         \[\sqcap\{\logic{L}_i:i\in I\}=\mathsf{Taut}(\otimes\{\logic{L}_i:i\in I\})\]
         holds for any set of $\heartsuit$-logics $\{\logic{L}_i:i\in I\}$.\label{logicmeet}
     \end{proposition}
     For a proof, see \cite[Prop. 1.3]{BezhanishviliBETVSCR}.
    
    The main rule systems of interest in this paper are the following. On the $\sim$ side, the system $\logic{mHC}$ (the \emph{modalized Heyting calculus}) is defined as the least $\sim$-rule system containing $/\varphi$ whenever $\varphi$ is a theorem of the intuitionistic propositional calculus, as well as the rules
    \begin{align*}
        &/{\boxtimes(p\to q)}\to ({\boxtimes p} \to {\boxtimes q}),\\
        &/p\to \boxtimes p,\\
        &/{\boxtimes p}\to (q\lor (q\to p)).
    \end{align*}
    Among systems in $\lat{NExt}(\logic{mHC})$, the system $\logic{KM}$ (the \emph{Kuznetsov-Muravitsky rule system}) is of particular interest. It is defined as the least rule system in $\lat{NExt}(\logic{mHC})$ containing the Gödel-Löb rule:
    \[/({\boxtimes p}\to p)\to p.\]
    We refer to \cite{Esakia2006TMHCaCMEotIL,Muravitsky2014LKaB}
for more infromation about $\logic{mHC}$ and $\logic{KM}$.

    Now to the $\clm$-rule systems. The weakest such system we consider is the system $\logic{K4}$, defined as the least $\clm$-rule system containing $/\varphi$ whenever $\varphi$ is valid in all transitive Kripke frames. Next, the system $\logic{K4.Grz}$ is axiomatized over $\logic{K4}$ by the Grzegorczyk rule
    \[/\square{ (\square( p\to\square p)\to p)\to p}.\]
    And finally, we have the system $\logic{GL}$, which is axiomatized over $\logic{K4}$ by adding the classical version of the Gödel-Löb rule:
    \[/\square{({\square p}\to p)}\to \square p.\]
    We note that $\logic{K4.Grz}$ is non-reflexive counterpart of the $\logic{Grz}$. We refer to  \cite{Esakia2006TMHCaCMEotIL} for more details on $\logic{K4.Grz}$.

\subsection{Algebras}

We recall that a \emph{Heyting algebra} is a tuple $\alg{H}:=(H, \land, \lor, \to , 0, 1)$ whose $\to$-free reduct is a bounded distributive lattice and the equivalence 
\[c\leq a\to b\iff a\land c\leq b\]
holds for each $a, b, c\in H$. A \emph{frontal Heyting algebra} is a structure $\alg{H}:=(H, \land, \lor, \to , \boxtimes,  0, 1)$ whose $\boxtimes$-free reduct is a Heyting algebra and $\boxtimes$ satisfies the following equalities and inequalities for all $a, b\in H$:
\begin{align*}
    \boxtimes 1&=1,\\
    \boxtimes (a\land b)&=\boxtimes a\land \boxtimes b,\\
    a&\leq \boxtimes a,\\
    \boxtimes a&\leq b\lor (b\to a).
\end{align*}
A frontal Heyting algebra is called a \emph{fronton} if, in addition, the identity
\[(\boxtimes a\to a)\leq a\]
holds for all $a\in H$. We refer to \cite{CastiglioniEtAl2010OFHA, Esakia2006TMHCaCMEotIL} for more infromation on these structures. We write $\op{fHA}$ and $\op{Frt}$ for the classes of all frontal Heyting algebras and of all frontons respectively. 

On the classical side, a \emph{modal algebra} is a tuple $\alg{M}:=(M, \land, \neg, \square, 0, 1)$ whose $\square$-free reduct is a Boolean algebra and the identities
\[\square 1= 1,\qquad \qquad \square{ (a\land b)}=\square a\land \square b\]
hold for all $a, b\in M$. A modal algebra is called 
\begin{itemize}
    \item A \emph{$\logic{K4}$-algebra} when it satisfies the inequality \[\square a \leq \square \square a;\] 
    \item A \emph{$\logic{K4.Grz}$-algebra} when it satisfies the inequality 
    \[\square(\square(a\to \square a)\to a)\leq a;\]
    \item A \emph{Magari algebra} when, in addition to the above, it satisfies the inequality \[\square(\square a\to a)\leq \square a.\]
\end{itemize}
We write $\op{K4, K4.Grz}$ and $\op{Mag}$ for the classes of all modal algebras, of all $\logic{K4}$-algebras, of all $\logic{K4.Grz}$-algebras, and of all Magari algebras respectively. 

A \emph{universal class} is a class of algebras in the same signature closed under isomorphic copies, subalgebras and ultraproducts. A \emph{variety} is a universal class which, in addition, is closed under homomorphisms and direct products. When $\class{U}$ is a class of algebras in the same signature, we write $\lat{Uni}(\class{U})$ and $\lat{Var}(\class{U})$, respectively, for the sets of universal classes and of varieties contained in $\class{U}$. We note that both sets form complete lattices, where the meets coincide with intersection.

The algebras just introduced can be used to give sound and complete semantics for \emph{sim}- and \emph{clm}-rule systems. When $\alg{A}$ is an algebra, a \emph{valuation} is any mapping $V:\mathit{Frm}_\heartsuit\to A$ that commutes with all primitive connectives of $\mathit{Frm}_\heartsuit$, where $\heartsuit=\mathit{sim}$ if $\alg{A}$ is a frontal Heyting algebra and $\heartsuit=\mathit{clm}$ if $\alg{A}$ is a modal algebra. An algebra equipped with a valuation is called a \emph{model}. We say that a model $(\alg{A}, V)$ \emph{satisfies} a rule $\Gamma/\Delta$ when the following holds: if $V(\gamma)=1$ for all $\gamma\in \Gamma$, then $V(\delta)=1$ for some $\delta\in \Delta$. We notate this as $\alg{A}, V\models \Gamma/\Delta$. A rule $\Gamma/\Delta$ is \emph{valid} on a $\nu$-algebra $\alg{A}$ when $\alg{A}, V\models\Gamma/\Delta$ holds for all valuations $V$ on $\alg{A}$.  When this holds we write $\alg{A}\models \Gamma/\Delta$, otherwise we write $\alg{A}\not\models \Gamma/\Delta$ and say that $\alg{A}$ \emph{refutes} $\Gamma/\Delta$. We can extend this notion of validity to classes of $\nu$-algebras in the obvious way; we use similar notation in this case. 

For $\heartsuit\in \{\mathit{sim, clm}\}$, when $\logic{S}$ is a $\heartsuit$-rule system, we write $\op{Alg}(\logic{S})$ for the class of all algebras of the appropriate kind (frontal Heyting for $\heartsuit=\mathit{sim}$, modal otherwise) that validate every $\Gamma/\Delta\in \logic{S}$. Conversely, if $\class{U}$ is a class of frontal Heyting or modal algebras, we let $\op{ThR}(\class{U})$ be the set of all $\heartsuit$-rules that are valid in every member of $\class{U}$, with $\heartsuit=\mathit{sim}$ if $\class{U}$ is a class of frontal Heyting algebras, $\heartsuit=\mathit{clm}$ otherwise. By Birkhoff's theorem (see, e.g., \cite{BurrisSankappanavar1981ACiUA}), $\op{Alg}(\logic{S})$ is always a universal class, and indeed a variety when $\logic{S}$ is a logic. Conversely, $\op{ThR}(\class{U})$ is always a rule system, and indeed a logic when $\class{U}$ is a variety. 
\begin{theorem}
    The following pairs of mappings are complete dual lattice isomorphisms:\label{algebraization}
    \begin{enumerate}
        \item $\op{Alg}:\lat{NExt}(\logic{mHC})\to \lat{Uni}(\op{fHA})$ and $\op{ThR}:\lat{Uni}(\op{fHA})\to \lat{NExt}(\logic{mHC})$;\label{compunimhc}
         \item $\op{Alg}:\lat{NExt}(\logic{K4})\to \lat{Uni}(\op{K4})$ and $\op{ThR}:\lat{Uni}(\op{K4})\to \lat{NExt}(\logic{K4})$;\label{compunik4}
         \item $\op{Alg}:\lat{NExtL}(\logic{mHC})\to \lat{Var}(\op{fHA})$ and $\op{ThR}:\lat{Var}(\op{fHA})\to \lat{NExtL}(\logic{mHC})$;\label{compvarmhc}
         \item $\op{Alg}:\lat{NExtL}(\logic{K4})\to \lat{Var}(\op{K4})$ and $\op{ThR}:\lat{Var}(\op{K4})\to \lat{NExtL}(\logic{K4})$.\label{compvark4}
    \end{enumerate}
\end{theorem}
\begin{proof}
    For \Cref{compunik4,compvark4}, see \cite[Thm. 2.2]{Jerabek2009CR} and \cite[Thm. 7.56]{ChagrovZakharyaschev1997ML}.  \Cref{compvarmhc} follows from \cite[Prop. 1]{WolterZakharyaschev2014OtBET}, and the proof of \Cref{compunimhc} is a routine generalization.
\end{proof}
\begin{corollary}
    Every $\sim$-rule system (resp.\ $\clm$-rule system) is complete with respect to a universal class of frontal Heyting algebras (resp.\ modal algebras). Likewise, every $\sim$-logic (resp.\ $\clm$-logic) is complete with respect to a variety of frontal Heyting algebras (resp.\ modal algebras).
\end{corollary}

It should be apparent that $\op{Frt}=\op{Alg}(\logic{KM})$, since a frontal Heyting algebra $\alg{H}$ satisfies the inequality ${\boxtimes a}\to a\leq a$ iff it validates the rule $/({\boxtimes p}\to p)\to  p$. It is likewise apparent that $\op{K4}=\op{Alg}(\logic{K4})$, $\op{K4.Grz}=\op{Alg}(\logic{K4.Grz})$ and $\op{Mag}=\op{Alg}(\logic{GL})$.

\subsection{Spaces and Kripke Frames}

We will also interpret rule systems over expansions of Stone spaces and over Kripke frames. 
When $X$ is a set and $U\subseteq X$, we write $X\smallsetminus U$ for the complement of $U$ in $X$. If $X$ is clear from context, we write $-U$ instead of $X\smallsetminus U$. 
Given a binary relation $\prec$ on a set $X$, for any $U\subseteq X$ we define
\begin{align*}
    \upset[\prec]{U}:=\{x\in X: y\prec x\text{ for some }y\in U\},\\
    \downset[\prec]{U}:=\{x\in X:x\prec y\text{ for some }y\in U\}.
\end{align*}
In the case where $U=\{y\}$ we write $\upset[\prec]{x}$ and $\downset[\prec]{x}$ instead of $\upset[\prec]{\{x\}}$ and $\upset[\prec]{\{x\}}$. We omit subscripts when the relation is clear from context. 

We recall that a \emph{Stone space} is a compact Hausdorff space with a basis of clopens. A \emph{modal space} is a triple $\spa{X}=(X, R, \class{O})$ such that $(X, \class{O})$ is a Stone space and $R$ is a binary relation satisfying the following conditions:
\begin{enumerate}
        \item $\upset{x}$ is closed for every $x\in X$;\label{esa1}
		\item $\downset{U}\in \clop{X}$ for every $U\in \clop{X}$. \label{esa2}
\end{enumerate}

A \emph{modalized Esakia space} is a quadruple $\spa{X}=(X, \leq, \sqsubset, \class{O})$ such that $(X, \class{O})$ is a Stone space and $\leq, \sqsubset$ are binary relations satisfying the following conditions:
	\begin{enumerate}
        \item $\leq$ is reflexive and transitive;
        \item $\upset[\leq]{x}$ is closed for every $x\in X$;
		\item $\downset[\leq]{U}\in \clop{X}$ for every $U\in \clop{X}$; 
		\item $\{x\in X: \upset[\sqsubset]{x}\subseteq U\}\in \clopup[\leq]{X}$ whenever $U\in \clopup[\leq]{X}$; 
		\item The reflexive closure of $\sqsubset$ coincides with $\leq$. 
	\end{enumerate}
We refer to \cite{CastiglioniEtAl2010OFHA} and \cite{Esakia2006TMHCaCMEotIL} for more infromation about  modalized Esakia spaces. 

A \emph{modal space} is a triple $\spa X=(X, R, \class{O})$ such that $(X, \class{O})$ is a Stone space and $R$ is a binary relation on $X$ satisfying the following conditions:
	\begin{enumerate}
        \item $\upset[R]{x}$ is closed for every $x\in X$;
		\item $\downset[R]{U}\in \clop{X}$ for every $U\in \clop{X}$.
	\end{enumerate}
In any modal space, we let $R^+$ denote the reflexive closure of $R$.

When $\spa{X}, \spa{Y}$ are modalized Esakia spaces, a mapping $f:\spa{X}\to \spa{Y}$ is called a \emph{bounded morphism} when it is continuous and both the conditions below hold whenever $\prec\in \{\leq, \sqsubset\}$.
\begin{enumerate}
    \item For all $x, y\in X$, if $x\prec y$,  then $f(x)\prec f(y)$;
    \item For all $x\in X$ and $y\in Y$, if $f(x)\prec y$, then there is some $z\in X$ such that $Rxz$ and $f(z)=y$.
\end{enumerate}
Likewise, when $\spa{X}, \spa{Y}$ are modal spaces, a mapping $f:\spa{X}\to \spa{Y}$ is called a \emph{bounded morphism} when it is continuous and both the conditions above hold for $\prec=R$. 

The categories of modalized Esakia spaces and of modal spaces, with their respective bounded morphisms, are dual to those of frontal Heying algebras and modal algebras with homomorphisms. Both these dualities are generalizations of the celebrated Stone duality between Boolean algebras and Stone spaces. Given an algebra $\alg{A}$, we construct its dual space $\dualspa{A}$ as follows. As the carrier we take the set of all prime filters on $\alg{A}$. A basis of clopens for the topology is given by all the sets of prime filters of the form $\beta(a)$ for some $a\in A$, where $\beta(a)$ is the set of all prime filters to which $a$ belongs. When $\alg{A}$ is a frontal Heyting algebra, the intuitionistic relation $\leq$ on the dual coincides with prime filter inclusion, whereas the modal relation is given by
\[x\sqsubset y:\iff \boxtimes a\in x \text{ implies } a\in y \text{ for all }a\in A.\]
When $\alg{A}$ is a modal algebra, the modal relation $R$ is defined exactly like $\sqsubset$, but substituting $\square$ for $\boxtimes$. In both cases, if $h:\alg{A}\to \alg{B}$ is a homomorphism, the dual bounded morphism $h_*:\dualspa{B}\to \dualspa{A}$ coincides with the preimage mapping $h^{-1}$. 

Conversely, if $\spa{X}$ is a space, its dual algebra $\dualalg{X}$ is constructed as follows. When $\spa{X}$ is a modalized Esakia space, the carrier is $\clopup[\leq]{X}$. The Heyting implication $\to$ and modal operator $\boxtimes$ of $\dualalg{X}$ are given by the identities 
\begin{align*}
    U\to V&:=-\downset[\leq]{(U\smallsetminus V)},\\
    \boxtimes U&:=\{x\in X: \upset[\sqsubset]{x}\subseteq U\}.
\end{align*}
On the other hand, when $\spa{X}$ is a modal space, the carrier of $\dualalg{X}$ is $\clop{X}$ and the modal operator $\square$ is defined the same way as $\boxtimes$, but substituting $R$ for $\sqsubset$. In both cases, when $f:\spa{X}\to \spa{Y}$ is a bounded morphism, the dual homomorphism $f^*:\dualalg{Y}\to \dualalg{X}$ again coincides with the preimage mapping  $f^{-1}$. Using these constructions we can prove the following duality result. 
\begin{theorem}
    The following categories are dually equivalent:\label{duality}
    \begin{enumerate}
        \item Frontal Heyting algebras with homomorphisms and modalized Esakia spaces with bounded morphisms \cite[Thm. 4.11]{CastiglioniEtAl2010OFHA};
        \item Modal algebras with homomorphisms and modal spaces with bounded morphisms \cite{jonsson1951boolean}.
    \end{enumerate}
\end{theorem}

 A \emph{sim-Kripke frame} is defined as a triple $\spa{X}=(X, \leq, \sqsubset)$ that satisfies all the non-topological conditions from the definition of a modalized Esakia space. Likewise, a \emph{clm-Kripke frame} is a triple $\spa{X}=(X, R)$ that satisfies all the non-topological conditions from the definition of a modal space. It will be convenient to treat Kripke frames as equipped with the discrete topology. Thus, when $\spa{X}$ is a Kripke frame, $\clop{X}$ refers to $\wp(\spa{X})$ and $\clopup[\prec]{X}$ refers to all the upsets (relative to $\prec$) in $\spa{X}$. Thus the two notions of bounded morphism defined for spaces readily apply to Kripke frames. 

 The constructions underwriting \Cref{duality} can be adapted to obtain duality results connecting Kripke frames and algebras of appropriate kinds. In the algebras-to-frames direction, the main change is to take the set of \emph{completely join prime filters} instead of the set of prime filters as the carrier. Accordingly, the Stone map $\beta$ is replaced with the mapping $\alpha$ that sends an element of an algebra to the set of completely join prime filters containing it. The other direction is essentially unchanged, given the conventions just mentioned about treating Kripke frames as equipped with the discrete topology. 
 \begin{theorem}
		The following categories are dually equivalent: \label{kripkedual}
		\begin{enumerate}
			\item $\sim$-Kripke frames with  bounded morphisms and  complete, completely distributive and completely join prime generated Heyting  algebras with complete homomorphisms; \label{kripkedual-h}
            \item $\clm$-Kripke frames with bounded morphisms and complete, atomic, completely additive and completely distributive modal algebras with complete homomorphisms.\label{kripkedual-m}
		\end{enumerate}
	\end{theorem}
 This result was first proved in \cite{Thomason1975CoFfML} for $\clm$-Kripke frames, see also \cite{Litak2008SotBT}. That it holds for $\sim$-Kripke frames appears to be a new observation, though the proof is a straightforward consequence of a similar result due to de Jongh \cite{Jongh1966OtCoPOSwSPBA} for intuitionistic Kripke frames. 

A \emph{sim-interpretation} is a mapping $V:\mathit{Frm}_{\mathit{sim}}\to \clopup[\leq]{X}$, where $\spa{X}$ is a modalized Esakia space or a \emph{sim}-Kripke frame, that commutes the right way with connectives, in particular:
\begin{align*}
    V(\varphi\to \psi)&=-\downset[\leq]{(V(\varphi)\smallsetminus V(\psi))},\\
    V(\boxtimes \varphi)&:=\{x\in X: \upset[\sqsubset]{x}\subseteq V(\varphi)\}.
\end{align*}
A \emph{clm-interpretation} is a mapping $V:\mathit{Frm}_{\mathit{clm}}\to \clop{X}$, where $\spa{X}$ is a modal space or a \emph{clm}-Kripke frame, that commutes the right way with connectives, in particular:
\[ V(\square \varphi)=\{x\in X: \upset[R]{x}\subseteq V(\varphi)\}.\]
In either case we write $\spa{X}, V, x\models \varphi$ to mean that $x\in V(\varphi)$, and $\spa{X}, V\models \varphi$ to mean that $\spa{X}, V, x\models \varphi$ holds for all $x\in X$. We write $\spa{X}, V\models \Gamma/\Delta$ to mean the following: if $\spa{X}, V\models \gamma$ for each $\gamma\in \Gamma$, then $\spa{X}, V\models \delta$ for \emph{some} $\delta\in \Delta$. A rule $\Gamma/\Delta$ is said to be \emph{valid} in a space or Kripke frame $\spa{X}$ when $\spa{X}, V\models \Gamma/\Delta$ holds for all valuations of the appropriate sort, otherwise it is \emph{refuted}. We write $\spa{X}\models\Gamma/\Delta$ to mean that $\Gamma/\Delta$ is valid in $\spa{X}$, and $\spa{X}\not\models\Gamma/\Delta$ to mean it is refuted in $\spa{X}$. The notions of validity and refutability can be straightforwardly generalized to classes of spaces or Kripke frames; we use similar notation for these cases.

\subsection{Properties of Transitive Structures}

We list here some basic properties of spaces and Kripke frames we will appeal to throughout the paper. Let $\spa X$ be a $\logic{K4.Grz}$-space and $U\in \op{Clop}(\spa X)$. We define:
\begin{align*}
    \max{U}&:=\{x\in U: \upset[R]{x}\cap U\subseteq\{x\}\},\\
    \pas{U}&:=\{x\in U: \text{ if $Rxy$ and $y\notin U$, then $\upset[R]{y}\cap U=\varnothing$}.\}
\end{align*}
We call elements of $\max{U}$ \emph{maximal} in $U$, and elements of $\pas{U}$ \emph{passive} in $U$.\footnote{The terminology of ``passive'' points is due to Esakia \cite{Esakia2019HADT}.} Intuitively the maximal elements of $U$ are those that do not ``see'' elements of $U$ other than, at most, themselves, and the passive elements of $U$ are those elements of $U$ from which it is impossible to ``leave'' and ``re-enter'' $U$. We define analogous notions for modalized Esakia spaces, substituting $\sqsubset$ for $R$. 

\begin{proposition}
    Let $\spa X$ be a $\logic{K4.Grz}$-space and let $U\in \clop{X}$. Then $\max{U}$ is closed and $\pas{U}\in \clop{X}$.\label{grzclosed}
\end{proposition}
\begin{proof}
    Follows from {\cite[Thm. 3.2.1, 3.5.5]{Esakia2019HADT}}.
\end{proof}
\begin{proposition}
    Let $\spa X$ be a modalized Esakia space. For all $U\in \clop{X}$ and any $x\in X$, if $\upset[\leq]{x}\cap U\neq\varnothing$, then $\upset[\leq]{x}\cap \max{U}\neq\varnothing$. Consequently, if $\upset[\sqsubset]{x}\cap U\neq\varnothing$, then $\upset[\sqsubset]{x}\cap \max{U}\neq\varnothing$.
    The claim remains true if we let $\spa X$ be a $\logic{K4.Grz}$-space and substitute $R^+$ for $\leq$ and $R$ for $\sqsubset$.\label{seemaximals}
\end{proposition}
\begin{proof}
    In both versions of the statement, the first part follows from \cite[Thm. 3.2.3]{Esakia2019HADT}. For the second part, assume  $\upset[\sqsubset]{x}\cap U\neq\varnothing$. If $x\in \max U$, then $\upset[\sqsubset]{x}=\{x\}$ and we are done, so suppose otherwise. We have $\upset[\leq]{x}\cap U\neq\varnothing$, so $\upset[\leq]{x}\cap \max U\neq\varnothing$ by the first part. This means there is $y\in \max U$ with $x\leq y$. Since $x\notin \max U$, we have $x\neq y$, and so $x\sqsubset y$. The modal case is analogous. 
\end{proof}

\begin{proposition}
    The following conditions hold:
    \begin{enumerate}
        \item A modalized Esakia space $\spa{X}$ is a $\logic{KM}$-space iff no point $x\in \max{U}$ is such that $x\sqsubset x$, for every $U\in \clopdown[\leq]{X}$;
        \item A $\logic{K4.Grz}$-space $\spa{X}$ is a $\logic{GL}$-space iff no point $x\in \max{U}$ is such that $Rx x$, for every $U\in \clop{X}$.
    \end{enumerate} \label{maxirr}
\end{proposition}
\begin{proof}
    \emph{$\logic{KM}$ case}. $(\Rightarrow)$ Let $\spa X$ be a modalized Esakia space, let $U\in \clopdown{X}$ and assume $x\in \max{U}$. Suppose towards a contradiction that $x\sqsubset x$. Let $V(p)=-U$. Then $V(p)\in \clopup{X}$. If $y\in \upset[\sqsubset]{x}$, then either $x=y$, in which case $\spa X, V, y\not\models \boxtimes p$, or $x\neq y$, so both $\spa X, V, y\models \boxtimes p$ and $\spa X, V, y\models  p$. Either way, $\spa X, V, x\models \boxtimes p\to p$, yet $\spa X, V, x\not\models p$. This contradicts the $\logic{KM}$ axiom $(\boxtimes p\to p)\to p$.

    $(\Leftarrow)$ Assume $\spa X$ is not a $\logic{KM}$-space. Then there is a valuation $V$ and a point $x\in X$ such that  $\spa X, V, x\models \boxtimes p\to p$ and $\spa X, V, x\not\models p$. This can only happen if $x\not\sqsubset x$. 
    
    \emph{$\logic{GL}$ case}. $(\Rightarrow)$  Let $\spa X$ be a $\logic{GL}$-space, let $U\in \clop{X}$ and assume $x\in \max{U}$. Suppose towards a contradiction that $Rxx$. Define $V(p)=U$. Then $\spa X, V, x\models \lozenge p$. However, if $y\in \upset[R]{x}$, then either $x=y$, in which case  $\spa X, V, y\not\models \square \neg p$, or $x\neq y$, in which case $\spa X, V, y\not\models p$. Either way, $\spa X, V, y\not\models \square \neg p\land p$, showing $\spa X, V, x\not\models \lozenge (\square \neg p\land p)$. This contradicts the $\logic{GL}$-axiom $\square (\square p\to p)\to \square p$.

    $(\Leftarrow)$ Assume $\spa X$ is not a $\logic{GL}$-space. Then there is a valuation $V$ and a point $x\in X$ such that  $\spa X, V, x\models \square(\square p\to p)$ and $\spa X, V, x\not\models\square p$. This can only happen if $Rx x$ fails. 
\end{proof}

\begin{proposition}
    Let $\spa X$ be a Kripke frame for either $\logic{KM}$ or $\logic{GL}$. Then the modal relation in $\spa X$ is conversely well-founded (hence irreflexive.) \label{maxirrkripke}
\end{proposition}
\begin{proof}
    The arguments are very similar to those given above. See, e.g.,  \cite[Cor. 3]{Litak2014CMwPS} for the $\logic{KM}$-case and \cite[Ex. 3.9]{BlackburnEtAl2001ML}.
\end{proof}

Lastly, when $\spa X$ is a space and $\prec$ is its modal relation, a non-empty set $C\subseteq X$ is called a \emph{cluster} when it is maximal with the property that, if $x, y\in C$, then both $x\prec y$ and $y\prec x$. A set $U\subseteq X$ is said to \emph{cut} a cluster $C\subseteq X$ when neither $C\subseteq U$ nor $C\cap U=\varnothing$ holds. A cluster is called \emph{proper} when its cardinality is greater than one, \emph{improper} otherwise.

\section{Filtration and Pre-filtration}

So much for preliminaries. We now move on to discussing our notion of \emph{pre-filtration}. As noted earlier, pre-filtration is motivated by the fact that standard filtration does not work well for $\logic{KM}$ and $\logic{GL}$. So let us begin by reviewing the standard notion of filtration. 

\subsection{Stable and Pre-stable Maps}

We begin by introducing some key definitions needed for our discussion of filtration and pre-filtration. Since we are mainly interested in the correspondence between $\sim$- and $\clm$-rule systems, we will henceforth restrict attention to $\logic{K4}$-algebras in the $\clm$ case. 

Throughout the paper we will often think of $\logic{K4}$-algebras as bimodal algebras in a signature expanded with an extra modal operator $\square^+$, defined by $\square^+:=\square a\land a$ \cite{Litak2014CMwPS}. This perspective will help us emphasize the connection between $\logic{K4}$-algebras and frontal Heyting algebras. Note that the identity
\[\square a={\square^+}\square \square^+ a\label{mix}\tag{\emph{mix}}\]
always holds in any $\logic{K4}$-algebra. 

\begin{definition}[Stable embedding]
    Let $\alg{A, B}$ be either frontal Heyting algebras or $\logic{K4}$-algebras. An injection $h:\alg{A}\to \alg{B}$ is called a \emph{stable embedding} when the following conditions hold:
    \begin{enumerate}
        \item \emph{Frontal Heyting case}: $h$ is a bounded distributive lattice embedding and $h(\boxtimes a)\leq \boxtimes h(a)$ holds for each $a\in A$;
        \item \emph{$\logic{K4}$-case}: $h$ is a Boolean embedding and $h(\square a)\leq \square h(a)$ holds for each $a\in A$.
    \end{enumerate}
\end{definition}
Note that whenever $h:\alg{H}\to \alg{K}$ is a distributive lattice embedding between frontal Heyting algebras, the inequality $h(a\to b)\leq h(a)\to h(b)$ always holds for all $a, b\in H$. Moreover,  whenever $h:\alg{M}\to \alg{N}$ is a stable embedding between $\logic{K4}$-algebras, the inequality $h(\square^+ a)\leq \square^+ h(a)$ always holds. Thus, a stable embedding partially preserves both $\to$ and $\boxtimes$ in the $\sim$-case, and both $\square^+$ and $\square$ in the $\clm$-case. 

Let $\alg{A}$ be an algebra. A \emph{unary domain} on $\alg{A}$ is simply a finite subset $D\subseteq A$, whereas a \emph{binary domain} on $\alg{A}$ is a finite subset $D\subseteq A\times A$.
\begin{definition}[Bounded domain condition]
    Let $\alg{A, B}$ be either frontal Heyting algebras or $\logic{K4}$-algebras. When $\odot$ is a unary or binary operator on $A$ (primitive or compound), we say that a map $h:\alg{A}\to \alg{B}$ satisfies the \emph{$\odot$-bounded domain condition} (BDC$^\odot$) for a domain $D$ of appropriate arity when the identities
\begin{align*}
    h(\odot a)&=\odot h(a)      &\text{if $\odot$ is unary},\\
    h({a{\odot} b})&={h(a){\odot }h(b)  }   &\text{otherwise}
\end{align*}
hold for every element of $D$. 
\end{definition}
In other words, $h$ satisfies the BDC$^\odot$ for a domain when it fully preserves $\odot$ on elements that belong to the domain. 

We now give a dual description of stable embeddings and the BDC. 
\begin{definition}[Stable map, dual]
    Let $\spa{X, Y}$ be modalized Esakia spaces, $\logic{K4}$-spaces, $\sim$- or $\clm$-Kripke frames. A map $f:\spa{X}\to \spa{Y}$ is called \emph{stable} when it is continuous and preserves all the primitive relations of $\spa{X}$---so both $\leq$ and $\sqsubset$ in the $\sim$ case and $R$ in the $\clm$ case. 
\end{definition}
The requirement that stable maps preserve $\leq$ is strictly speaking redundant, since any map between modalized Esakia spaces or $\sim$-Kripke frames that preserves $\sqsubset$ automatically preserves $\leq$. Furthermore, in the $\clm$ case, it is clear that stable maps also preserve the reflexive closure $R^+$ of $R$. 

If $\spa{X}$ is a modalized Esakia space or $\sim$-Kripke frame, a \emph{unary domain} on $\spa{X}$ is a finite subset $\spa{D}\subseteq\clopdown[\leq]{X}$, whereas a \emph{binary domain} is a finite subset $\spa{D}\subseteq\clop{X}$ where each $\spa{d}\in \spa{D}$ is of the form $U\cap -V$, with both $U, V\in \clopup[\leq]{X}$. Equivalently, each $\spa{d}$ must equal the intersection of a clopen upset with a clopen downset. On the other hand, if $\spa{X}$ is a transitive modal space or $\clm$-Kripke frame, a \emph{unary domain} on $\spa{X}$ is just a finite subset $\spa{D}\subseteq\clop{X}$.
\begin{definition}[Bounded domain condition, dual]
    Let $\spa{X, Y}$ be spaces or Kripke frames, let $\prec$ be any binary relation on $\spa{X}$ and let $\spa{D}$ be a unary or binary domain on $\spa{Y}$. We say that a stable map $f:\spa{X}\to \spa{Y}$ satisfies the \emph{$\prec$-bounded domain condition} (BDC$^\prec$) for $\spa{D}$ when the following holds: for each $x\in X$ and any $\spa{d}\in \spa{D}$, if there is some $y\in \spa{d}$ such that $f(x)\prec y$, then there must be some $z\in X$ such that $x\prec z$ and $f(z)\in \spa{d}$.
\end{definition}
Notice that when $\spa{Y}$ is finite and $\spa{D}$ consists precisely of the singletons of points in $Y$, $f$ is stable and satisfies the BDC$^\prec$ for $\spa{D}$ precisely when it is a bounded morphism with respect to the relation $\prec$, whence the name for the condition.

Here is the key duality result relating the concepts just introduced. 
\begin{proposition}\label{stabledual}
    The following are equivalent, for all frontal Heyting algebras $\alg{H, K}$ and all $\logic{K4}$-algebras $\alg{M, N}$:
    \begin{enumerate}
        \item A map $h:\alg{H}\to \alg{K}$ is a stable embedding satisfying the BDC$^\to$ for $D^\to$ and the BDC$^\boxtimes$ for $D^\boxtimes$ iff $h^{-1}:\dualspa{K}\to \dualspa{H}$ is a stable surjection satisfying the BDC$^\leq$ for $D^\to_*$ and the BDC$^\sqsubset$ for $D^\boxtimes_*$, where 
        \begin{align*}
            D^\to_*&:=\{\beta (a)\cap -\beta(b):(a, b)\in D^\to\},\\
            D^\boxtimes_*&:=\{-\beta (a): a\in D^\square\}.
        \end{align*}
        Moreover,  the same holds when we substitute $\dualkr{H}$, $\dualkr{K}$ and $\alpha$  for $\dualspa{H}$, $\dualspa{K}$ and $\beta$ respectively, when the former are well defined.
        \item A map $h:\alg{M}\to \alg{N}$ is a stable embedding satisfying the BDC$^\square$ for $D^\square$ iff $h^{-1}:\dualspa{N}\to \dualspa{M}$ is a stable surjection satisfying the BDC$^\square$ for $D^\square_*$ , where \[D^\square_*:=\{-\beta (a): a\in D^\boxtimes\}.\]  
        Moreover,  the same holds when we substitute $\dualkr{M}$, $\dualkr{N}$ and $\alpha$ for $\dualspa{M}$ and $\dualspa{N}$ and $\beta$ respectively, when the former are well defined.
    \end{enumerate}
\end{proposition}
A proof of the modal case for spaces  is given in \cite[Sec. 3]{BezhanishviliBETVSCR}, and can be straightforwardly adapted to cover the remaining cases. 

We will also need slight variants of the concepts just introduced, to be used in defining our notion of pre-filtration. The first two are weakenings of the notions of a stable embedding and of a stable map. 
\begin{definition}[Pre-stable embedding]
    Let $\alg{A, B}$ be either frontal Heyting algebras or $\logic{K4}$-algebras. An injection $h:\alg{A}\to \alg{B}$ is called a \emph{pre-stable embedding} when the following conditions hold:
    \begin{enumerate}
        \item \emph{Frontal Heyting case}: $h$ is a bounded distributive lattice embedding;
        \item \emph{$\logic{K4}$-case}: $h$ is a Boolean embedding and $h(\square^+ a)\leq \square^+ h(a)$ holds for each $a\in A$.
    \end{enumerate}
\end{definition}
Since, as noted already, every bounded distributive lattice embedding partially preserves $\to$ and every Boolean embedding that partially preserves $\square$ also partially preserves $\square^+$, pre-stable embeddings differ from stable embeddings only  in that they do not require the partial preservation of $\boxtimes$ and $\square$, depending on the signature.

We note the following simple fact. 
\begin{proposition}
    Let $\alg{M, N}$ be $\logic{K4}$-algebras, let $h:\alg M\to \alg N$ a pre-stable embedding  and let $D\subseteq N$. If $h$ satisfies the BDC$^\square$ for $D$, then it satisfies the BDC$^{\square^+}$ for $D$.\label{simplefact}
\end{proposition}
\begin{proof}
    Immediate from the fact that $\square^+a:=\square a\land a$ and the fact that pre-stable embeddings are Boolean embeddings. 
\end{proof}

\begin{definition}[Pre-stable map, dual]
    Let $\spa{X, Y}$ be modalized Esakia spaces, $\logic{K4}$-spaces, $\sim$- or $\clm$-Kripke frames. a map $f:\spa{X}\to \spa{Y}$ is called \emph{pre-stable} when it is continuous and preserves $\leq$ (in the $\sim$ case) and $R^+$ (in the $\clm$ case). 
\end{definition}
Again, pre-stable maps differ from stable maps only in that the preservation of $\sqsubset$ and $R$ is not required. 

We need one last definition, which is a slight strengthening of the BDC. 
\begin{definition}[Back and forth condition]
     Let $\spa{X, Y}$ be spaces or Kripke frames, let $\prec$ be any binary relation on $\spa{X}$ and let $\spa{D}$ be a unary or binary domain on $\spa{Y}$. We say that a pre-stable map $f:\spa{X}\to \spa{Y}$ satisfies the \emph{back and forth condition} (BFC$^\prec$) for $\spa{D}$ when the following conditions hold for each $x\in X$ and any $\spa{d}\in \spa{D}$:
     \begin{description}
         \item[\textbf{Back}:] If there is some $y\in \spa{d}$ such that $f(x)\prec y$, then there must be some $z\in X$ such that $x\prec z$ and $f(z)\in \spa{d}$;
         \item[\textbf{Forth}:] if there is $y\in f^{-1}(\spa{d})$ with $x\prec y$, then there must be some $z\in \spa{d}$ with $f(x)\prec z$.
    \end{description}
\end{definition}
In other words, $f$ satisfies the BDC$^\prec$ for $\spa{D}$ when $\upset[\prec]{f(x)}\cap \spa{d}\neq \varnothing$ holds iff $f[\upset[\prec]{x}]\cap\spa{d}\neq\varnothing$, for all $x\in X$ and $\spa{d}\in \spa{D}$. Note that every stable map that satisfies the BDC$^{\prec}$ for $\spa{D}$ automatically satisfies the BFC$^\prec$ for the same domain, since the requirement that a stable map preserve $\prec$ implies the second item in the definition of the BFC$^{\prec}$. However, of course, not every pre-stable maps that satisfies the BFC$^\prec$ for a domain is stable.

A duality result analogous to \Cref{stabledual} can be established using essentially the same argument. 
\begin{proposition}\label{prestabledual}
    The following are equivalent, for all frontal Heyting algebras $\alg{H, K}$ and all $\logic{K4}$-algebras $\alg{M, N}$: \label{bdcdual}
    \begin{enumerate}
        \item A map $h:\alg{H}\to \alg{K}$ is a pre-stable embedding satisfying the BDC$^\to$ for $D^\to$ and the BDC$^\boxtimes$ for $D^\boxtimes$ iff $h^{-1}:\dualspa{K}\to \dualspa{H}$ is a pre-stable surjection satisfying the BFC$^\sqsubset$ for  $D^\boxtimes_*$ and the the BFC$^\leq$ for  $D^\boxtimes_*$, with $D^\to_*$ and $D^\boxtimes_*$ defined as before. 
        Moreover,  the above remains true when we substitute $\dualkr{H}$, $\dualkr{K}$ and $\alpha$  for $\dualspa{H}$, $\dualspa{K}$ and $\beta$ respectively, when the former are well defined;
        \item A map $h:\alg{M}\to \alg{N}$ is a stable embedding satisfying the BDC$^{\square^+}$ for $D^{\square^+}$ and the BDC$^\square$ for $D^\square$ iff $h^{-1}:\dualspa{N}\to \dualspa{M}$ is a stable surjection satisfying the BFC$^{R^+}$ for $D^{\square^+}_*$ and the BFC$^R$ for $D^{\square}_*$, with $D^{\square}_*$ defined as before and $D^{\square^+}_*$ defined in analogous fashion. 
        Moreover,  the same holds when we substitute $\dualkr{M}$, $\dualkr{N}$ and $\alpha$ for $\dualspa{M}$ and $\dualspa{N}$ and $\beta$ respectively, when the former are well defined.
    \end{enumerate}
\end{proposition}
\begin{proof}
    The argument is very similar to that given in the proof of \Cref{stabledual}. The only difference lies in showing that $h$ satisfies the BFC$^\boxtimes$ for $D^\boxtimes$ iff $h_*$ satisfies the BFC for $D^\boxtimes_*$, and the counterpart claim in the $\clm$ case. This is routine. 
\end{proof}

Before we move forward, we note a property of pre-stable maps over $\logic{GL}$-spaces, which will come useful later.
\begin{lemma}
    Let $\spa X, \spa Y$ be $\logic{GL}$-spaces and let $f:\spa X\to \spa Y$ be a pre-stable map satisfying the BFC$^R$ for some $\spa D$. Then $f^{-1}(\max{\spa d})=\max{f^{-1}(\spa d)}$ for each $\spa d\in \spa D$.\label{maxmax}
\end{lemma}
\begin{proof}
    Assume $f(x)\notin \max{\spa d}$. Suppose $\upset[R]{f(x)}\cap \spa d\neq\varnothing$. Since $f$ satisfies the BFC for $\spa d$, we must have $\upset[R]{x}\cap f^{-1}(\spa d)\neq \varnothing$, which implies $x\notin \max{f^{-1}(\spa d)}$ by \Cref{maxirr}. The other direction is analogous, but using the ``forth'' condition from the BFC. 
\end{proof}

\subsection{Shortcomings of Standard Filtration}

We take an algebraic approach to defining filtration, following \cite{BezhanishviliEtAl2016SCR, BezhanishviliEtAl2016CSL}. 
\begin{definition}
    Let $\alg{A}$ be an algebra, $V$ a valuation on $\alg{M}$, and $\Theta$ a finite, subformula closed set of formulae (in the appropriate signature). A finite model $(\alg{N}, V')$ is called a  \emph{filtration of $(\alg{M}, V)$ through $\Theta$} if the following conditions hold:
    \begin{enumerate}
        \item \emph{$\sim$ case}
        \begin{enumerate}
				\item The bounded distributive lattice reduct of $\alg{B}$ is isomorphic to the bounded distributive sublattice  of $\alg{A}$ generated by  $V[\Theta]$;
				\item $V(p)=V'(p)$ for every propositional variable $p\in \Theta$;
				\item The inclusion $\subseteq:\alg{B}\to \alg{A}$ is a stable embedding satisfying the BDC$^\to$ for the domain $\{(V(\varphi), V(\psi)): \varphi\to \psi\in \Theta\}$ and the BDC$^\boxtimes$ for the domain $\{V(\varphi):\boxtimes \varphi\in \Theta\}$.
			\end{enumerate}
        \item \emph{$\clm$ case}
        \begin{enumerate}
				\item The Boolean algebra reduct of $\alg{B}$ is isomorphic to the Boolean subalgebra of $\alg{A}$ generated by $V[\Theta]$;
				\item $V(p)=V'(p)$ for every propositional variable $p\in \Theta$;
				\item The inclusion $\subseteq:\alg{B}\to \alg{A}$ is a stable embedding satisfying the BDC$^\square$ for the set $\{V(\varphi):\square \varphi\in \Theta\}$.
			\end{enumerate}
        \end{enumerate}
\end{definition}
The definition of filtration in the $\clm$ case is completely standard. The literature on $\sim$-rule systems and logics is far more limited than that on $\clm$-rule systems and logics, to the point that it may not make much sense to speak of a standard definition of filtration in this setting. But the definition just given is as standard as it can be: it is obtained by simply conjoining the standard definition of filtration for models based on modal algebras with the standard definition of models for superintuitionistic logics based on Heyting algebras. 

By \Cref{stabledual}, it follows that a finite model $(\alg{B}, V')$ is a filtration of a model $(\alg{A}, V)$ through some set of formulae only if there is a stable surjection from $\dualspa{A}$ to $\dualspa{B}$. This shows why filtration does not work well for either Magari algebras or $\logic{KM}$-algebras: some $\logic{GL}$-spaces have \emph{no} image under a stable surjection which is also a $\logic{GL}$-space, and the same holds true for $\logic{KM}$-spaces. 

 Considering the case of $\logic{GL}$ first, an example is the space $\spa{X}$ depicted in \Cref{nofiltration}. This space is constructed as follows:
\begin{itemize}
		\item $X=\mathbb{N}\cup \{\omega_0, \omega_1\}$;
		\item $R=\{(n, m)\in \mathbb{N}\times \mathbb{N}: m<n\}\cup \{(\omega_i, x)\in X\times X: i\in \{0, 1\}\text{ and }x\in X\}$;
		\item $\mathcal{O}$ is given by the basis consisting of all $U\subseteq X$ such that either $U$ is a finite subset of $\mathbb{N}$, or $U=V\cup \{\omega_0,\omega_1\}$ with $V$ a cofinite subset of $\mathbb{N}$, or $U$ is one of the following sets:
		\[\{n\in \mathbb{N}:n\text{ is even}\}\cup \{\omega_0\}\quad \quad \{n\in \mathbb{N}:n\text{ is odd}\}\cup \{\omega_1\}.\]
\end{itemize}
\begin{figure}
\centering
   \begin{tikzpicture}
		\node[irr] (0) [label=right:{$0$}] {};
		\node[irr] (1) [below=of 0, label=right:{$1$}] {};
		\node[irr] (2) [below=of 1, label=right:{$2$}] {};
		\node (dots) [below=of 2] {$\vdots$};
		\node[roundnode] (w0) [below left = of dots, label=below:{$\omega_0$}] {};
		\node[roundnode] (w1) [below right = of dots, label=below:{$\omega_1$}] {};
		
		\draw[->] (1) edge (0);
		\draw[->] (2) edge (1); 
		\draw[->] (w0) edge[bend right=40] (w1);
		\draw[->] (w1) edge[bend right=40] (w0);
	\end{tikzpicture}
    \label{nofiltration}
    \caption{The $\logic{GL}$-space $\spa{X}$}
\end{figure}

This space contains a non-trivial cluster made up of two reflexive points, but it is nonetheless  a $\logic{GL}$-space. To see this, one need only observe that any $U\in \clop{X}$ that contains at least one of $\omega_1, \omega_2$ must also contain some element  $n\in\mathbb{N}$. Indeed, any $x\in \max{U}$ must belong to $\mathbb{N}$, which consists entirely of irreflexive points. By \Cref{maxirr}, this implies that $\spa{X}$ is a $\logic{GL}$-space. Now, we know from \Cref{maxirrkripke} that no finite $\logic{GL}$ space can contain reflexive points. Yet \emph{every} image of $\spa{X}$ under a stable map must contain at least one reflexive point, namely the image of $\omega_1$ or of $\omega_2$. Thus, no filtration of any model based on $\spa{X}$ is based on a $\logic{GL}$-space. 

It is worth remarking on why this shows that filtration does not work well for $\logic{GL}$ and its extensions. A paradigm application of the filtration construction is to prove finite model property results. To prove that a rule system $\logic{M}$ has the finite model property using filtration, one takes a model $\logic{M}$ that refutes a rule $\Gamma/\Delta$ and then constructs a filtration of that model through $\mathit{Sfor}(\Gamma/\Delta)$. The filtration always refutes $\Gamma/\Delta$, but the argument only suffices to establish the finite model property if the filtration is also a model of $\logic{M}$. Examples like the one just given therefore show that filtration is not a good tool for proving that $\logic{GL}$ or some extension thereof has the finite model property. 

These remarks generalize to the case of $\logic{KM}$. Consider, for example, the $\logic{KM}$-space obtained by collapsing the lower cluster from the space $\spa{X}$ in \Cref{nofiltration}, letting the modal relation be as before (modulo the collapse of the cluster) and letting the intuitionistic relation be the reflexive closure of the modal relation. The resulting space is depicted in \Cref{nofiltrationkm}. It has a reflexive point under the modal relation, so any surjective image thereof under a relation preserving map must have one as well. But no finite $\logic{KM}$-space can contain reflexive points under the modal relation. 
\begin{figure}
\centering
   \begin{tikzpicture}
		\node[irr] (0) [label=right:{$0$}] {};
		\node[irr] (1) [below=of 0, label=right:{$1$}] {};
		\node[irr] (2) [below=of 1, label=right:{$2$}] {};
		\node (dots) [below=of 2] {$\vdots$};
		\node[roundnode] (w0) [below  = of dots, label=below:{$\omega_0$}] {};
		
		\draw[->] (1) edge (0);
		\draw[->] (2) edge (1); 
	\end{tikzpicture}
    \label{nofiltrationkm}
    \caption{The $\logic{KM}$-space $\spa{Y}$}
\end{figure}

\subsection{Pre-filtration}

We have discussed the shortcomings of standard notions of filtration for normal extensions of $\logic{GL}$ and $\logic{KM}$. We now introduce \emph{pre-filtration} as a more general notion of filtration intended to overcome these shortcomings. 

\subsubsection{The $\sim$ case}

It is instructive to begin with the $\sim$-case. The standard algebraic definition of filtration for Heyting algebras is motivated by a simple uniqueness result: every finite distributive lattice $\alg{L}$ has a \emph{unique} expansion to a Heyting algebra. This is done by defining a Heyting implication on $\alg{L}$ as
\[a\to b=\bigvee \{c\in L: a\land c\leq b\}.\]
It is this uniqueness result that allows one to prove that every model of a superintuitionistic logic based on a Heyting algebra has a filtration through any subformula-closed set of formulae $\Theta$. One starts by taking a model $(\alg{H}, V)$, then one generates a bounded distributive lattice from the $V[\Theta]$. The resulting lattice must be finite, because bounded distributive lattices are locally finite, and it can be expanded to a Heyting algebra $\alg{K}$ using the construction just sketched. The inclusion $\subseteq: \alg{K}\to \alg{H}$ happens to be a bounded distributive lattice embedding that satisfies the BDC$^\to$ for the domain consisting of valuations of formulae $\varphi, \psi$ such that $\varphi\to \psi\in \Theta$. 

This suggests a way forward. The correct way of defining a suitable notion of filtration for frontons is not simply to combine the standard modal and superintuitionistic definitions of filtration. Rather, we first need to look for a suitable uniqueness result  that can be used to generate finite frontons from infinite ones, then model our new notion of filtration after the properties of this generation procedure. 

Luckily, such a uniqueness result exists. 
\begin{theorem}[{\cite[Proposition 5]{Esakia2006TMHCaCMEotIL}}]
    A Heyting algebra $\alg{H}$ can be expanded to a fronton iff the set  
    \[F_a:=\{b\in H: b\to a\leq b\}\]
    is a  principal proper filter for each $a\in H$. Moreover, there is a unique such expansion.
    \label{frontonsquare}
\end{theorem}
\begin{proof}
    $(\Rightarrow)$ Assume $\alg{H}$ has the structure of a fronton. It is routine to check that $F_a$ is a proper filter; we show that it is principal. Now, for each $a\in H$ we have $\boxtimes a\to a\leq a\leq \boxtimes a$, whence $\boxtimes a\in F_a$. Given $b\in F_a$, we have $b\lor (b\to a)\leq b$ and $\boxtimes a\leq b\lor (b\to a)$, whence $\boxtimes a\leq b$. 

   $(\Leftarrow)$ Assume that each $F_a$ is a principal proper filter. Define
   \[\boxtimes a:=\bigwedge F_a.\]
   Thus $\boxtimes a$ coincides with the generator of $F_a$. It is routine to check that the result of expanding $\alg{H}$ with $\boxtimes$ is a frontal Heyting algebra. Moreover, it is also a fronton. For by $\boxtimes a\in F_a$ we have $\boxtimes a\to a\leq \boxtimes a$. This is to say $(\boxtimes a\to a)\land \boxtimes a=\boxtimes a\to a$. Since, clearly, $(\boxtimes a\to a)\land \boxtimes a\leq a$, it follows that $\boxtimes a\to a\leq a$, as desired. 

   From the arguments just given it follows that every fronton must satisfy the identity $\boxtimes a=\bigwedge F_a$. This establishes the uniqueness condition in the theorem. 
\end{proof}
\begin{corollary}
    Every finite Heyting algebra has a unique expansion to a fronton. \label{uniquefinitefronton}
\end{corollary}
\begin{proof}
   In a finite Heyting algebra $\alg{H}$, each filter $F_a$ must contain the meet $\bigwedge F_a$. Thus we can expand $\alg{H}$ to a fronton using the construction described in the proof of the $(\Leftarrow)$ direction of the previous theorem. By the same theorem, this expansion is unique. 
\end{proof}

There is thus a unique way of expanding a finite distributive lattice to a fronton. Every finite distributive lattice has a unique expansion to a Heyting algebra, which in turn has a unique expansion to a fronton. Composing these two construction yields the mapping we use to model our new notion of pre-filtration. 
\begin{definition}[Pre-filtration, $\sim$ case]
	Let $\mathfrak{H}$ be a frontal Heyting algebra, $V$ a valuation on $\mathfrak{H}$, and $\Theta$ a finite, subformula closed set of formulae. A finite model $(\mathfrak{K}, V')$, with $\mathfrak{K}\in \mathsf{fHA}$, is called a  \emph{pre-filtration of $(\mathfrak{H}, V)$ through $\Theta$} if the following hold: \label{def-prefiltration}
	\begin{enumerate}
		\item The $(\boxtimes, \to)$-free reduct of $\alg{K}$ is isomorphic to the bounded sublattice of $\mathfrak{H}$ generated by a finite superset of $ V[\Theta]$;\footnote{Note we now require only that the reduct be generated by a finite superset of $\bar V[\Theta]$ rather than by $\bar V[\Theta]$ itself. This slight generalization is intended to make the statements and proofs of some of the theorems to follow simpler; nothing substantial hinges on this. }
		\item $V(p)=V'(p)$ for every propositional variable $p\in \Theta$;
		\item The inclusion $\subseteq:\mathfrak{K}\to \mathfrak{H}$ is a pre-stable embedding satisfying the BDC$^\to$ for the set $\{( V(\varphi),  V(\psi)):\varphi\to \psi\in \Theta\}$, and satisfying the BDC$^\boxtimes$ for the set $\{ V(\varphi):\boxtimes \varphi\in \Theta\}$.
	\end{enumerate}
\end{definition}
The only difference between the definitions of filtration and pre-filtration is that we have dropped the requirement that $\boxtimes$ be partially preserved. This is key, as the construction we use to extract finite frontons, based on the proof of \Cref{uniquefinitefronton}, does not partially preserve $\boxtimes$. This illustrates how we can use pre-filtration to get around problem cases like the space $\spa{Y}$ in \Cref{nofiltrationkm}. A pre-stable map need not preserve the modal relation $\sqsubset$, so finite $\logic{KM}$-spaces can be surjective images of $\spa{Y}$ under pre-stable maps despite the latter containing  the $\sqsubset$-reflexive point $\omega_0$.

The uniqueness result from \Cref{frontonsquare} can be used to establish the existence of ``enough'' pre-filtrations, in the following sense. 
\begin{theorem}
    Let $\alg{H}$ be a fronton. For every model $(\alg{H}, V)$ and any subformula-closed set of formulae $\Theta$, there is a pre-filtration $(\alg{K}, V')$ of $(\alg{H}, V)$ through $\Theta$ such that $\alg{K}$ is a fronton. \label{admitsprefiltrationsim} 
\end{theorem}
\begin{proof}
    Let $\alg{K}_0$ be the bounded sublattice of $\alg{H}$ generated by $V[\Theta]$ and let $D^\boxtimes:=\{V(\varphi):\boxtimes \varphi\in \Theta\}$. Let $a_1, \ldots, a_k$ enumerate $D^\boxtimes$ and define recursively:
    \begin{align*}
        C_{i+1}:=\{(b\to a_{i+1})\land \boxtimes a_{i+1}:b\in K_i \cap [a_{i+1}, \boxtimes a_{i+1}]\},
    \end{align*}
    where $[a, \boxtimes a]:=\{b\in H:a\leq b\leq \boxtimes a\}$. Finally, let $\alg{K}_{i+1}$ be the bounded sublattice of $\alg{H}$ generated by $K_i\cup C_i$.

    To get an intuitive grip on this construction, recall \citep[Lemma 4]{Castiglioni2012TFMPftVoHAwS} that in any fronton $\alg{H}$ the interval $[a, \boxtimes a]$, viewed as a sublattice of $\alg{H}$, is a Boolean algebra where the complement $\neg_a b$ of any $b\in [a, \boxtimes a]$ is given by 
    \[\neg_a b:=(b\to a)\land \boxtimes a.\]
    Thus, the move from $\alg{K}_i$ to $\alg{K}_{i+1}$ consists in adding complements to elements of $\alg K_i$ relative to the interval $[a_{i+1}, \boxtimes a_{i+1}]$, then generating a new distributive lattice. 

    Since each $\alg{K}_i$ is finite, it can be viewed as a fronton by defining on it a Heyting implication $\rightsquigarrow_i$ and a modal operator $\otimes_i$ as in the proof of \Cref{frontonsquare}. Let $\alg{K}:=\alg{K}_k$ and $\otimes:=\otimes_k$. By construction, the inclusion embedding $\subseteq:\alg K\to \alg H$ is pre-stable. We show that it satisfies the BDC for $(D^\to, D^\boxtimes)$, where $D^\to:=\{(V(\varphi), V(\psi)): \varphi\to \psi\in\Theta\}$.

    Since
    \[a\rightsquigarrow b:=\bigvee \{c\in K:a\land c\leq b\},\qquad a\to b:=\bigvee \{c\in H:a\land c\leq b\},\]
    if $a\to b\in V[\Theta]\subseteq K$, then $a\to b\leq a\rightsquigarrow b$, whence $a\rightsquigarrow b=a\to b$ given pre-stability. Now let $a_i\in D^\boxtimes$. 
    Observe that $\otimes_i(a_i)=\boxtimes(a_i)$. For note that
    \[\otimes_i a_i:=\bigwedge \{b\in K:b\rightsquigarrow_i a_i\leq a_i\},\qquad {\boxtimes a_i}:=\bigwedge \{b\in H:b\to a_i \leq a_i\}.\]
     Since $\boxtimes a_i\in D^\boxtimes \subseteq K$ and ${\boxtimes a_i\rightsquigarrow_i a_i}\leq \boxtimes a_i\to a_i$, it follows that $\otimes_i a_i\leq \boxtimes a_i$. Consequently, $\otimes_i a_i\in K_i\cap [a_i, \boxtimes a_i]$. 
   Now, $\otimes_i a_i\rightsquigarrow_i a_i=a_i$ holds because $\alg K_i$ is a fronton. Moreover, given that $\boxtimes a_i, a_i, \boxtimes a_i\to a_i\in V[\Theta]$, we also have ${\boxtimes a_i\rightsquigarrow a_i}=\boxtimes a_i\to a_i=a_i$. So:
    \begin{align*}
        \neg_{a_i} {\otimes_i a_i}&=(\otimes_i a_i\rightsquigarrow_i a_i)\land \boxtimes a_i=a_i\land \boxtimes a_i=a_i,\\
        \neg_{a_i} {\boxtimes a_i}&=(\boxtimes a_i\rightsquigarrow_i a_i)\land \boxtimes a_i=a_i\land \boxtimes a_i=a_i.
    \end{align*}
    Because Boolean complements are unique, it follows that $\otimes_i a_i=\boxtimes a_i$. By analogous reasoning we have $\otimes_i a_i= \otimes_{i+1} a_i$. Putting these observations together, we conclude that $\otimes a= \boxtimes a$ holds for every $a\in D^\boxtimes$, as desired. 

    Finally, define a valuation $V'$ on $\alg{K}$ by putting  $V'(p)=V(p)$ if $p\in \Theta$, arbitrary otherwise. The resulting model $(\alg{K}, V')$ is a pre-filtration of $(\alg H, V)$ through $\Theta$.
\end{proof}
The argument just presented is a close adaptation of the proof of \cite[Thm. 6]{Castiglioni2012TFMPftVoHAwS}, which is related to a construction used by Muravitsky  \cite[Thm. 2]{Muravitskii1981FAotICatEoaEHNM} to establish the finite model property for $\logic{KM}$.

Despite the weaker definition, pre-filtrations still preserve and reflect the satisfaction of rules whose premises and conclusions belong to the ``filtering set'' $\Theta$.
\begin{theorem}[Pre-filtration theorem for frontal Heyting algebras]\label{prefiltrationtheorem}
    Let $\mathfrak{H}$ be a fronton, $V$ a valuation on $\mathfrak{H}$, and $\Theta$ a finite, subformula-closed set of formulae. If $(\mathfrak{K}, V')$ is a pre-filtration of $(\mathfrak{H}, V)$ through $\Theta$ then for every $\varphi\in \Theta$ we have
	\[ V(\varphi)= V'(\varphi).\]
	Consequently, for every rule $\Gamma/\Delta$ such that $\gamma, \delta\in \Theta$ for each $\gamma\in \Gamma$ and $\delta\in \Delta$ we have 
	\[\mathfrak{H}, V\models \Gamma/\Delta\iff \mathfrak{K}', V'\models \Gamma/\Delta.\]
\end{theorem}
\begin{proof}
    A straightforward induction on the structure of formulae. 
\end{proof}

We conclude with some remarks about the existence of pre-filtrations of frontal Heyting algebras more generally. Given $\alg H, V$  and $\Theta$ as in \Cref{def-prefiltration}, the expansion construction from \Cref{frontonsquare} can always be used to extract a finite fronton generated as a bounded sublattice of $\alg{H}$ by $V[\Theta]$. But of course, in general, the construction cannot be continued to construct a pre-filtration of $(\alg{H}, V)$ based on $\alg{K}$. For if $\Theta$ contains $(\boxtimes p\to p)\to p$ and $\alg{H}$ is not a fronton, $(\alg{H}, V)$  and $(\alg{K}, V')$ will disagree on whether $(\boxtimes p\to p)\to p$ is valid, so $(\alg{K}, V')$ cannot be a pre-filtration of $(\alg{H}, V)$ through $\Theta$. 

It thus remains an open question whether the version of \Cref{admitsprefiltrationsim} obtained by substituting ``frontal Heyting algebra'' for ``fronton'' is true. This is an instance of a broader problem: it is, in general, not clear how to construct filtrations in signatures containing multiple non-interdefinable operators.

\subsubsection{For Modal Algebras}

Let us now turn to the $\clm$ case. The key intuition we appeal to here, which we mentioned already, is that every $\logic{K4}$-algebra can be seen as a bimodal algebra in a signature with an extra modal operator $\square^+$, defined as
\[\square^+a:=a\land \square a.\]
We have defined pre-filtration for frontal Heyting algebras by means of pre-stable embeddings, which  partially preserve $\to$ but not necessarily $\boxtimes$. Likewise, to define pre-filtration for $\logic{K4}$-algebra, we use embeddings that partially preserve $\square^+$ but not necessarily $\square$ itself, namely pre-stable embeddings again.

\begin{definition}[Pre-filtration for $\logic{K4}$-algebras]
	Let $\mathfrak{M}$ be a $\logic{K4}$-algebra, $V$ a valuation on $\mathfrak{M}$, and $\Theta$ a finite, subformula closed set of formulae. A finite model $(\mathfrak{N}, V')$, with $\mathfrak{N}$ a $\logic{K4}$-algebra, is called a  \emph{pre-filtration of $(\mathfrak{M}, V)$ through $\Theta$} if the following hold:
	\begin{enumerate}
		\item The $\square$-free reduct of $\alg{N}$ is the Boolean subalgebra of $\alg{M}$ generated by a finite superset of $V[\Theta]$;
		\item $V(p)=V'(p)$ for every propositional variable $p\in \Theta$;
		\item The inclusion $\subseteq:\mathfrak{N}\to \mathfrak{M}$ is a pre-stable embedding satisfying the BDC$^{\square}$ for the set $\{ V(\varphi):\square \varphi\in \Theta\}$.\label{prefiltr-3}
	\end{enumerate}
\end{definition}
One might have thought \Cref{prefiltr-3} should be strengthened so that $\subseteq$ is also required to satisfy the BDC$^{\square^+}$ for the set $D^{\square^+}:=\{ V(\varphi):\square^+ \varphi\in \Theta\}$. But this is not really a strengthening. By the definition of $\square^+$, this set is contained in $D^\square:=\{ V(\varphi):\square \varphi\in \Theta\}$. Moreover, by \Cref{simplefact}, it already follows from \Cref{prefiltr-3} that $\subseteq$ satisfies  BDC$^{\square^+}$ for $D^{\square}$, hence for $D^{\square^+}$ as well.

As before, the weaker definition does not affect the ability of pre-filtrations to preserve and reflect the satisfaction of formulae in the filter set $\Theta$.
\begin{theorem}[Pre-filtration theorem for $\logic{K4}$-algebras]
    Let $\mathfrak{M}$ be a $\logic{K4}$-algebra, $V$ a valuation on $\mathfrak{H}$, and $\Theta$ a a finite, subformula-closed set of formulae. If $(\mathfrak{N}, V')$ is a pre-filtration of $(\mathfrak{M}, V)$ through $\Theta$ then for every $\varphi\in \Theta$ we have
	\[ V(\varphi)= V'(\varphi).\]
	Consequently, for every rule $\Gamma/\Delta$ such that $\gamma, \delta\in \Theta$ for each $\gamma\in \Gamma$ and $\delta\in \Delta$ we have 
	\[\mathfrak{M}, V\models \Gamma/\Delta\iff \mathfrak{N}, V'\models \Gamma/\Delta.\]
\end{theorem}

The existence of a pre-filtration of a model based on a $\logic{K4}$-algebra through any subformula-closed set of formulae is an immediate consequence of the existence of a filtration of that model through that set of formulae, since every filtration is a pre-filtration.
\begin{theorem}
     Let $\alg{M}$ be a $\logic{K4}$-algebra. For every model $(\alg{M}, V)$ and any subformula-closed set of formulae $\Theta$, there is a pre-filtration $(\alg{N}, V')$ of $(\alg{M}, V)$. \label{admitsfiltrationclm}
\end{theorem}

This result, however, is not strong enough to support the use of pre-filtration in, say, proofs to the effect that some normal extension of $\logic{GL}$ has the finite model property. To that end, we would need to strengthen the consequent to the effect that $\alg{N}$ be a Magari algebra. While we were not able to confirm this, we conjecture that this strengthening of \Cref{admitsfiltrationclm} is false. However, later in the paper we will be able to establish a slightly weaker result (\Cref{admitsfiltrationmagari}), which is nonetheless strong enough to support finite model property arguments for normal extensions of $\logic{GL}$ using pre-filtration.


\section{Pre-stable Canonical Rules}

We have introduced our notion of pre-filtration. We now turn to the task of syntactically encoding pre-filtrations through algebra-based rules. When $\alg{A}$ is an algebra, for each $a\in A$ introduce a fresh propositional variable $p_a$.

\begin{definition}[Pre-stable canonical rule, $\sim$ case]
		Let $\alg{H}$ be a finite frontal Heyting algebra and let $D:=(D^\to, D^\boxtimes)$ be respectively a binary and a unary domain on $\alg{H}$. The \emph{$\sim$ pre-stable canonical rule} $\scrsi{H}{D}$ is defined as the rule $\Gamma/\Delta$, where \label{def:prescrsim}
		\begin{align*}
			\Gamma=&\{p_0\leftrightarrow \bot\}\cup\{p_1\leftrightarrow \top\}\cup\\
			&\{p_{a\land b}\leftrightarrow p_a\land p_b:a, b\in H\}\cup \{p_{a\lor b}\leftrightarrow p_a\lor p_b:a, b\in H\}\cup\\
			& \{p_{a\to b}\leftrightarrow p_a\to p_b:(a, b)\in D^\to\} \cup \{p_{\boxtimes a}\leftrightarrow \boxtimes p_a: a\in D^\boxtimes\}\\
			\Delta=&\{p_a\leftrightarrow p_b:a, b\in H\text{ with } a\neq b\}.
		\end{align*}
	\end{definition}
	\begin{definition}[Pre-stable canonical rule, $\clm$ case]
		Let $\alg{M}$ be a finite $\logic{K4}$-algebra and let $D$ be a unary domain on $\alg{M}$.  The  \emph{$\clm$ stable canonical rule}  $\scrmod{A}{D}$ is defined as the rule $\Gamma/\Delta$, where \label{def:prescrclm}
		\begin{align*}
			\Gamma=&\{p_0\leftrightarrow \bot\}\cup\{p_1\leftrightarrow \top\}\cup\\
			&\{p_{a\land b}\leftrightarrow p_a\land p_b:a, b\in M\}\cup \{p_{\neg a}\leftrightarrow \neg p_a:a \in M\}\cup\\
			& \{p_{\square^+ a}\to \square^+ p_a:a\in M\} 	\cup \{p_{\square a}\leftrightarrow \square p_a:a\in D\} 			\\
			\Delta=&\{p_a:a\in A\smallsetminus 1\}.
		\end{align*}
	\end{definition}
In both cases, a pre-stable canonical rule \emph{fully} represents the truth functional structure of an algebra (i.e., the lattice structure of a frontal Heyting algebra and the Boolean structure of a $\logic{K4}$-algebra), but only represents non-truth functional structure on the distinguished domains. We use the notation $\scr{A}{D}$ to refer to either a $\sim$ or a $\clm$ pre-stable canonical rule, without specifying which.

Pre-stable canonical rules are useful because they have well behaved refutation conditions, connected to the pre-filtration construction. The next two results describe this connection in both its algebraic and dual version. Henceforth, when $D$ is a pair of domains on an algebra $\alg{A}$ as above, say that a map $h:\alg{A}\to \alg{B}$ satisfies the \emph{bounded domain condition} for $D$ when it does for both the domains occurring in $D$. 
\begin{proposition}
    Let $\alg{A}$ be a frontal Heyting or $\logic{K4}$-algebra and let $\scr{B}{D}$ be a pre-stable canonical rule of the appropriate kind. Then $\alg{A}\not\models \scr{B}{D}$ iff there is a pre-stable embedding $h:\alg{B}\to \alg{A}$ satisfying the BDC for $D$.\label{refutalg}
\end{proposition}
\begin{proof}
    We show the $\sim$ case only, as the $\clm$ case is analogous. $(\Rightarrow)$ Suppose $\alg{A}\not\models \scr{B}{D}$ and take a valuation $V$ that witnesses this. Define a map $h:\alg{B}\to \alg{A}$ by putting $h(a):=V(p_a)$. This is a bounded lattice embedding, because the model $(\alg{A}, V)$ satisfies every formula in $\Gamma$. For the same reason, it satisfies the BDC for $D$. In the case of $D^\to$, for $(a, b)\in D^\to$ note 
    \begin{align*}
        h(a\to b)&=V(p_{a\to b}) &\\
        &=V(p_a\to p_b) &\text{By $\alg{A}, V\models p_{a\to b}\leftrightarrow (p_a\to p_b)$}\\
        &=V(p_a)\to V(p_b) &\\
        &=h(a)\to h(b).
    \end{align*}
    Likewise in the case of $\boxtimes$. 

    $(\Leftarrow)$ Let $h:\alg{B}\to \alg{A}$ be a pre-stable embedding satisfying the BDC for $D$. Define a valuation $V$ on $\alg{A}$ by setting $V(p_a):=h(a)$. Because $h$ is pre-stable, the model $(\alg{A}, V)$ satisfies all formulae in the first two lines from the definition of $\Gamma$ in \Cref{def:prescrsim}. From the fact that $h$ satisfies the BDC for $D$, on the other hand, we can show that the reamining formulae in $\Gamma$ are satisfied, by essentially reversing the reasoning summarized in the identities above. Finally, $a\neq b$ implies $\alg{A}, V\not\models a\leftrightarrow b$, so no formula in $\Delta$ is satisfied in $(\alg{A}, V)$. We have thus shown that $\alg{A}\not\models \scr{B}{V}$.
\end{proof}
\begin{proposition}
    Let $\spa{X}$ be a modalized Esakia or $\logic{K4}$-space and let $\scr{A}{D}$ be a pre-stable canonical rule of the appropriate kind. Then $\spa{X}\not\models \scr{A}{D}$ iff there is a pre-stable surjection $f:\spa{X}\to \dualspa{A}$ that satisfies the BDC for ${D}_*$.\footnote{See the statement of \Cref{stabledual} for the definitions of ${D}_*$.}  Moreover, the same remains true if we let $\spa{X}$ be a $\sim$- or $\clm$-Kripke frame and substitute $\dualkr{A}$ for $\dualspa{A}$. \label{refutspace}
\end{proposition}
\begin{proof}
    Follows immediately from \Cref{refutalg,prestabledual}.
\end{proof}
In view of \Cref{refutspace}, we adopt the convention of writing a pre-stable canonical rule $\scr{A}{D}$ as $\scr{\dualspa{A}}{D_*}$ when working with spaces or Kripke frames. Note that since any such $\alg{A}$ is finite, $\dualkr{A}$ is always defined and equals $\dualspa{A}$. 

We are now ready to prove the main result that underwrites the usefulness of pre-stable canonical rules: that \emph{every} rule is equivalent (over a suitable base system) to a finite conjunction of pre-stable canonical rules. 
\begin{theorem}
    Let $\Gamma/\Delta$ be either a $\sim$-rule or $\clm$-rule. The following conditions hold:\label{rewrite}
    \begin{enumerate}
        \item \emph{$\sim$ case}: there is a finite set $\Phi$ of $\sim$ pre-stable canonical rules based on frontons, such that a fronton $\alg{H}$ refutes $\Gamma/\Delta$ iff it refutes some rule $\scrsi{K}{D}\in \Phi$;
        \item \emph{$\clm$ case}: there is a finite set $\Phi$ of $\clm$ pre-stable canonical rules such that a $\logic{K4}$-algebra $\alg{M}$ refutes $\Gamma/\Delta$ iff it refutes some rule $\scrsi{N}{D}\in \Phi$.
    \end{enumerate}
\end{theorem}
\begin{proof}
    \emph{$\sim$ case}: let $M(k)$ be the cardinality of the free bounded distributive lattice on $k$ generators. Let $M_0(k):=M(k)$ and $M_{i+1}(k)=M(2\cdot M_i(k))$. We let $\Phi$ be the set of all $\sim$ pre-stable canonical rules $\scrsi{K}{D}$ such that:
    \begin{enumerate}
        \item $\alg{K}$ is a fronton generated, as a bounded distributive lattice, by at most $M_n(k)$ elements, where $k=|\mathit{Sfor}(\Gamma/\Delta)|$ and $n=|\{\varphi:\boxtimes\varphi\in \mathit{Sfor}(\Gamma/\Delta)\}|$;
        \item $D=(D^\to, D^\boxtimes)$, where, for some valuation $V$ on $\alg{K}$ such that $\alg{K}, V\not\models \Gamma/\Delta$, the binary domain $D^\to$ consists precisely of the pairs $(V(\varphi), V(\psi))$ with $\varphi\to\psi\in \mathit{Sfor}(\Gamma/\Delta)$, and the unary domain $D^\boxtimes$ consists precisely of the elements $V(\varphi)$ with $\boxtimes \varphi\in \mathit{Sfor}(\Gamma/\Delta)$. 
    \end{enumerate}
    Each algebra $\alg{K}$ as above must be finite because bounded distributive lattices are locally finite, and since we have fixed the cardinality of the generators there are only finitely many such rules $\scrsi{K}{D}$, up to isomorphism of the underlying algebras. 

    $(\Rightarrow)$ Assume $\alg{H}\not\models \Gamma/\Delta$ and pick a witnessing valuation $V$. Construct a pre-filtration $(\alg{K}, V')$ of $(\alg{H}, V)$ through $\mathit{Sfor}(\Gamma/\Delta)$, which is always possible by \Cref{admitsprefiltrationsim}. In addition, note that the construction used to prove \Cref{admitsprefiltrationsim} ensures that $\alg K$ can be chosen to meet the cardinality constraints on the generators above. Use the valuation $V'$ to construct $D:=(D^\to, D^\boxtimes)$ as above. Since $(\alg{K}, V')$ is a filtration of $(\alg{H}, V)$, the inclusion $\subseteq:\alg{K}\to \alg{H}$ is a pre-stable embedding satisfying the BDC for $D$, showing $\alg{H}\not\models \scrsi{K}{D}$. By \Cref{prefiltrationtheorem} we have $\alg{K}, V'\not\models \Gamma/\Delta$, so $\scrsi{K}{D}\in \Phi$. 

    $(\Leftarrow)$ Assume $\alg{H}\not\models \scrsi{K}{D}$ for some $\scrsi{K}{D}\in \Phi$. Let $V$ be a valuation on $\alg{K}$ that can be used to construct $D$ as per the second item above. This can also be seen as a valuation on $\alg{H}$. Moreover, the model $(\alg{K}, V)$ is a filtration of $(\alg{H}, V)$ through $\mathit{Sfor}(\Gamma/\Delta)$. By \Cref{prefiltrationtheorem}, it follows that $\alg{H}, V\not\models \Gamma/\Delta$.

    \emph{$\clm$ case}: the proof is completely analogous, save for some minor adaptations in the construction of $\Phi$. In this case, we let $\Phi$ consist of all $\clm$ pre-stable canonical rules $\scrmod{N}{D}$ such that:
    \begin{enumerate}
        \item $\alg{N}$ is a $\logic{K4}$-algebra generated, as a Boolean algebra, by at most $|\mathit{Sfor}(\Gamma/\Delta)|$ elements;
        \item For some valuation $V$ on $\alg{N}$ such that $\alg{N}, V\not\models \Gamma/\Delta$, the unary domain $D$ consists precisely of the elements $V(\varphi)$ with $\square\varphi\in \mathit{Sfor}(\Gamma/\Delta)$.
    \end{enumerate}
    Since Boolean algebras are locally finite, each such $\alg{N}$ is finite, and $\Phi$ is finite because the number of generators has been fixed. 
\end{proof}
Note \Cref{rewrite} does not imply that every $\clm$ rule is equivalent \emph{over $\logic{GL}$} to finitely many pre-stable canonical rules \emph{based on Magari algebras}. To establish that result, we would need a stronger version of \Cref{admitsfiltrationclm}, establishing the existence of enough pre-filtration based on Magari algebras. As anticipated, a result of this sort holds true, though we will need to go through some theory of monomodal companions in the next section before we can establish it. 

\begin{corollary}
    The following conditions hold:
    \begin{enumerate}
        \item Every $\sim$-rule system is axiomatizable over $\logic{KM}$ by $\sim$ pre-stable canonical rules based on frontons;
        \item Every $\clm$-rule system above $\logic{K4}$ is axiomatizable over $\logic{K4}$ by $\clm$ pre-stable canonical rules.
    \end{enumerate}
\end{corollary}
\begin{proof}
    In either case, take an arbitrary axiomatization of the desired rule system and use \Cref{rewrite} to rewrite it in terms of pre-stable canonical rules. 
\end{proof}

The upshot of this section is that, in the right circumstances, we are now entitled to assume without loss of generality that we are always working with pre-stable canonical rules. The problem of checking whether a given rule is refuted by an algebra has been reduced to the problem of establishing whether pre-stable embeddings with the relevant properties exist.


%

\section{Monomodal Companions}

We now apply pre-stable canonical rules to the theory of monomodal companions of $\sim$-rule systems. 

\subsection{Mappings and Translations}

Recall that the \emph{free Boolean extension} of a Heyting algebra $\alg{H}$ is the unique Boolean algebra $B(\alg{H})$ in which $\alg{H}$ embeds as a distributive lattice, such that the image of $\alg{H}$ under this embedding generates $B(\alg{H})$ as a Boolean algebra [\citealp[Def.\ 2.5.6, Constr.\ 2.5.7]{Esakia2019HADT}; \citealp[Sec.\ V]{BalbesDwinger1975DL}]. For simplicity, we will generally identify $\alg{H}$ with its image in $B(\alg{H})$. Note this convention is used in the definitions to follow.

    If $\alg{H}$ is a frontal Heyting algebra, we define $\greatest\alg{H}$ by expanding $B(\alg{H})$ with the operation
    \[
    \square a:=\boxtimes Ia,
    \]
    where 
    \[
    Ia:=\bigvee\{b\in H:b\leq a\}.
    \]
    By the properties of free Boolean extensions, $Ia$ always exists and indeed belongs to $H$.
    Conversely, if $\alg{M}$ is a $\logic{K4}$-algebra we define $\fragment\alg{M}$ as follows. First, define the \emph{quasi-open elements} of $\alg{M}$
    \[O^+(\alg{M}):=\{a\in M: \square^+ a=a\}.\]
    This is a bounded distributive sublattice of $\alg{M}$. We let $\fragment\alg{M}$ be the result of expanding $O^+(\alg{M})$ with the operations 
    \begin{align*}
        a\to b&:=\square^+ (\neg a\lor b),\\
        \boxtimes a&:= \square a.
    \end{align*}
    Since $\square^+\square^+ a=\square^+ a$ and $\square^+\square\square^+ a=\square a$ for each $a\in M$, both operations are well defined. 

    We now give a dual description of these constructions. If $\spa{X}$ is a modalized Esakia space, let $\greatest\spa{X}$ be the $\leq$-free reduct of $\spa{X}$. To emphasize that we view $\greatest\spa X$ as a modal space, we will denote the remaining binary relation in $\greatest\spa X$ as $R$ instead of $\sqsubset$. Conversely, let $\spa{X}$ be a $\logic{K4}$-space. Define an equivalence relation on $\spa{X}$ by putting $x\backsim y$ iff either $x=y$ or both $Rxy$ and $Ryx$. Let $\varrho$ be the quotient map induced by $\backsim$. We define $\fragment\spa{X}$ by endowing $\varrho[X]$ with the quotient topology and expanding the resulting space with the binary relations
    \[\varrho (x)\leq \varrho (y):\iff R^+xy,\qquad\qquad \varrho (x)\sqsubset \varrho (y):\iff Rxy.\]
    These definitions do not depend on the choice of $x, y$: we could have equivalently defined $\varrho (x)\leq \varrho (y)$ iff there are $z\backsim x$ and $w\backsim y$ with $R^+zw$, and likewise for $\sqsubset$. These constructions are readily extended to Kripke frames, applying our convention of regarding a Kripke frame as endowed with the discrete topology. 

    \begin{proposition}
        The identities $\dualspa{(\greatest H)}=\greatest(\dualspa{H})$ and $\dualspa{(\fragment  M)}=\fragment(\dualspa{M})$ hold for every frontal Heyting algebra $\alg{H}$ and any $\logic{K4}$-algebra $\alg{M}$. Moreover, the second identity remains true if we replace $\dualkr{(\cdot)}$ for $\dualspa{(\cdot)}$, provided $\fragment\alg M$ is defined. 
    \end{proposition}
    \begin{proof}
        Note that $B(\alg H)$ is isomorphic to $\clop{\dualspa{H}}$, so $\dualspa{(\greatest H)}=\greatest(\dualspa{H})$ are isomorphic as Boolean algebras. Let $f:\dualspa{(\greatest H)}\to \greatest(\dualspa{H})$ be the bijection determined by the identity mapping on the Stone space dual to $B(\alg H)$. That $Rxy$ implies $Rf(x)f(y)$ is obvious. Conversely, assume $Rf(x)f(y)$ and take $U\in \clop{\dualspa{(\greatest H)}}$. Suppose $y\notin U$. By the definition of $\greatest\alg H$, 
        \[\square U=\boxtimes\bigcup\{V\in \clopup[\leq]{X}:V\subseteq U\}.\]
        Since $y\notin U$, also $\bigcup\{V\in \clopup[\leq]{X}:V\subseteq U\}$, so $x\notin \square U$. This shows $Rxy$. 

        For the second identity, first observe that given $x, y\in \dualspa M$ we have $\varrho (x)=\varrho (y)$ iff $x$ and $y$ contain the same quasi-open elements from $\alg M$. Since the quasi-open elements in a prime filter of $\alg M$ form a prime filter in $\fragment \alg M$ and every prime filter of $\fragment \alg M$ can be extended to a prime filter in $\alg M$ containing it, the map $f: \fragment (\dualspa M)\to \dualspa{(\fragment M})$ where $f(\varrho (x))$ is the prime filter on $\fragment \alg M$ consisting of all the quasi-open elements shared by all members of $\varrho(x)$ is a bijection. It should also be clear that $f$ preserves and reflects $\leq$. To see that it also preserves and reflects $\sqsubset$, let $\rho(x), \rho (y)\in \fragment (\dualspa M)$ and suppose $\rho(x)\sqsubset \rho(y)$. Take $a\in \fragment \alg M$ and suppose $a\notin f(\rho(y))$. Then also $a\notin \rho (y)$, so $\boxtimes a\notin \rho (x)$, and in turn $\boxtimes a\notin f(\rho (x))$. The other direction is analogous. 
    \end{proof}
    \begin{remark}
        It is important to point out that $\dualkr{H}$ being well defined does not guarantee that the identity $\dualkr{(\greatest H)}=\greatest(\dualkr{H})$ holds, even when both sides are well defined. As we shall see in a moment, $\greatest H$ is always a $\logic{K4.Grz}$-algebra, so $\dualkr{(\greatest H)}$ must be a $\logic{K4.Grz}$-frame if defined. However, $\greatest(\dualkr{H})$ need not be a $\logic{K4.Grz}$-frame, for it might not be conversely well-founded. 
    \end{remark}

    Using this dual description of the mappings $\greatest$ and $\fragment$, we can prove a few key facts about them. 
    \begin{proposition}
        The identity $\fragment \greatest \alg H=\alg H$ holds for every frontal Heyting algebra $\alg{H}$. \label{rhosigmaid}
    \end{proposition}
    \begin{proof}
        Dually, composing $\sigma$ with $\rho$ obviously yields the identity mapping.
    \end{proof}
    \begin{proposition}
        There is a modal algebra embedding of $\greatest\fragment\alg{M}$ into $\alg{M}$, for every $\logic{K4}$-algebra $\alg{M}$. \label{sigmarhosub}
    \end{proposition}
    \begin{proof}
        The cluster collapse map $\varrho:\dualspa{M}\to \greatest\fragment\dualspa{M}$ is a surjective bounded morphism, so its dual is a modal algebra embedding $\varrho_*:\greatest\fragment \alg M\to \alg M$. 
    \end{proof}
    \begin{proposition}
        Let $\alg{H}$ and $\alg{M}$ be a frontal Heyting algebra and a $\logic{K4}$-algebra respectively. Then $\greatest\alg H$ is a $\logic{K4.Grz}$-algebra and $\fragment \alg M$ is a frontal Heyting algebra. Moreover, if $\alg{H}$ is a fronton, then $\greatest\alg H$ is a Magari algebra, and if $\alg{M}$ is a Magari algebra then $\fragment \alg{M}$ is a $\logic{KM}$-algebra. 
    \end{proposition}
    \begin{proof}
        For a proof of the first part of this proposition, see \cite[Thm.\ 18]{Esakia2006TMHCaCMEotIL}. We prove the dual version of the second part. Let $\spa{X}$ be a $\logic{KM}$-space. Let  $U\in \clop{\greatest \spa X}$ and $x\in \max{U}$. Then  $\downset[\leq] U\in \clop{\greatest \spa X}$ and clearly also $x\in \max{\downset[\leq] U}$. Since $\downset[\leq] U$ is a clopen downset, by \Cref{maxirr} it follows that $x\not\sqsubset x$, which is to say that $Rx$ fails. By \Cref{maxirr} again, it follows that $\greatest{X}$ is a $\logic{GL}$-space. The last claim in the proposition proved similarly.
    \end{proof}

    We can extend both mappings $\greatest$ and $\fragment$ to universal classes by setting
    \[\greatest \class{U}:=\op{Uni}\{\greatest \alg H: \alg H\in \class{U}\},\qquad \fragment \class{V}:=\{\fragment \alg M: \alg M\in \class{V}\},\]
    whenever $\class{U}\in \lat{Uni}(\op{fHA})$ and $\class{V}\in \lat{Uni}(\op{K4}).$ We also introduce a mapping 
    \[\tau\class{U}:=\{\alg{M}\in \op{K4}: \fragment \alg{M}\in \class{U}\}.\]
    
    The syntactic counterpart of the mappings just introduced is the translation mapping $T:\mathit{Frm}_{\sim}\to \mathit{Frm}_{\clm}$ defined recursively below. 
    \begin{align*}
        T(\top)&:=\top,                                               &T(\bot)&:=\bot, \\
        T(p)&:=\square p,                                             &T(\varphi\lor \psi)&:=T(\varphi)\lor T(\psi), \\
         T(\varphi\land \psi)&:=T(\varphi)\land T(\psi),                 &T(\varphi\to \psi)&:=\square^+(\neg T(\varphi)\lor T(\psi)),  \\  
        T(\boxtimes \varphi)&:=\square T(\varphi).
    \end{align*}        
    $T$ as just defined is equivalent to the translation given in \cite{KuznetsovMuravitsky1986OSLAFoPLE,WolterZakharyaschev1997IMLAFoCBL}, noting that $\square^+\square\square^+\varphi\leftrightarrow \square \varphi$ is a theorem of $\logic{K4}$. We extend $T$ to a translation between rules by setting $T(\Gamma/\Delta):=T[\Gamma]/T[\Delta]$. 

    The key result required to work with $T$ is the following lemma. 
    \begin{lemma}
        Let $\varphi$ be a $\sim$-formula  and $\alg M$ a $\logic{K4}$-algebra. Then $\fragment \alg M\models \varphi$ iff $\alg M\models T(\varphi)$. Consequently, $\fragment \alg M\models \Gamma/\Delta$ iff $\alg M\models T(\Gamma/\Delta)$ holds for every $\sim$-rule $\Gamma/\Delta$.\label{translationrho}
    \end{lemma}
    \begin{proof}
        Straightforward induction on the structure of $\varphi$. 
    \end{proof}

    Using $T$, we define three mappings between lattices of rule systems:
    \[\least \logic{L}:=\logic{K4}\oplus \{T(\Gamma/\Delta):\Gamma/\Delta\in \logic{L}\},\qquad \greatest \logic{L}:=\logic{K4.Grz}\oplus \least \logic{L},\]
    \[\fragment \logic{M}:=\{\Gamma/\Delta\in \mathit{Rul}_{\sim}:T(\Gamma/\Delta)\in \logic{M}\},\]
    for $\logic{L}\in\lat{NExt}(\logic{mHC})$ and $\logic{M}\in \lat{NExt}(\logic{K4})$. 

    In general, the mappings $\tau$ and $\sigma$ over rule systems come apart. For example $\least\logic{mHC}=\logic{K4}\subset \logic{K4.Grz}=\greatest\logic{mHC}$. However, they happen to coincide above $\logic{KM}$.
    \begin{proposition}
        Let $\logic{L}\in \lat{NExt}(\logic{KM})$, which is to say $\logic{L}=\logic{KM}\oplus \Phi$ for some set of rules $\Phi$. Then \label{sigmataucoincide}
        \[\least \logic{L}=\logic{GL}\oplus T[\Phi]=\greatest \logic{L}.\]
    \end{proposition}
    \begin{proof}
        It suffices to observe that $\least \logic L\in \lat{NExt}( \logic{GL})$ whenever $\logic{L}\in \lat{NExt}(\logic{KM})$. To that end, note
        \[T((\boxtimes p\to p)\to p):=\square^+(\square^+(\square \square^+ p\to \square^+p)\to \square^+ p).\]
        Now, $/\square^+(\square \square^+ p\to \square^+p)$ implies $/\square(\square \square^+ p\to \square^+p)$ over $\logic{K4}$, which in turn implies $/\square^+ p$ over $\logic{GL}$ using the Löb formula. By necessitation and propositional reasoning it follows that $T(/(\boxtimes p\to p)\to p))$ is a theorem of $\logic{GL}$. 
    \end{proof}

\subsection{Main Results}

In this section, we give new proofs of the Esakia theorem and the Kuznetsov-Muravistsky isomorphisms. The next lemma is our main technical tool.  

\begin{lemma}[Main lemma]
    Let $\spa{X}$ be a $\logic{K4.Grz}$-space and let $\Gamma/\Delta$ be a $\clm$-rule. Then $\spa{X}\not\models \Gamma/\Delta$ iff $\greatest\fragment \spa{X}\not\models \Gamma/\Delta$.\label{mainlemma}
\end{lemma}
\begin{proof}
    The right-to-left direction is a consequence of \Cref{sigmarhosub} and the fact that the validity of rules is preserved under subalgebras. To prove the converse, by \Cref{rewrite} we may assume wlog that $\Gamma/\Delta$ is a $\clm$ pre-stable canonical rule $\scrmod{F}{\spa{D}}$ for some finite $\logic{K4}$-space $\spa{F}$. So suppose $\spa{X}\not\models \scrmod{F}{\spa{D}}$. Then there is a pre-stable surjection $f:\spa{X}\to \spa{F}$ satisfying the BFC for $\spa{D}$. We construct a pre-stable surjection $g:\greatest\fragment\spa X\to \spa F$ satisfying the BFC for the same domain, which will show $\greatest \fragment X\not\models \spa F$. 

    Let $C$ be a cluster in $\spa{F}$ and enumerate it $C:=c_1, \ldots, c_n$. For each $c_i\in C$, consider the preimage $f^{-1}(c_i)\subseteq X$ and let $M_i:=\mathit{max}(f^{-1}(c_i))$. Now, $f^{-1}(c_i)\in \clop{X}$, so $M_i$ is closed by \Cref{grzclosed}. Moreover, since $f$ preserves $R^+$, we know that $f^{-1}(c_i)$ does not cut clusters, so neither does $M_i$. Consequently, $-\varrho[M_i]=\varrho[-M_i]$, which implies that $\varrho[M_i]$ is closed as well because $\greatest\fragment \spa{X}$ has the quotient topology.

    We now find disjoint clopens $U_1, \ldots, U_n\in \clop{\greatest\fragment X}$ such that $\varrho[M_i]\subseteq U_i$ for each $i$, and $\bigcup_i U_i=\varrho[f^{-1}(C)]$. Let $k\leq n$ and assume $U_i$ has been defined for each $i<k$. If $k=n$,  put $U_n:=\varrho[f^{-1}(C)]\smallsetminus \left(\bigcup_{i<k} U_i\right)$ and we are done. Otherwise set $V_k:=\varrho[f^{-1}(C)]\smallsetminus\left(\bigcup_{i< k} U_i\right)$ and observe that $V_k$ contains each $\varrho[M_i]$ for  $k\leq i\leq n$. By the separation properties of Stone spaces, for each $i$ with  $k<i\leq n$ there is some $U_{k_i}\in \mathsf{Clop}(\greatest\fragment\spa{X})$ with $\varrho[M_k]\subseteq U_{k_i}$ and $\varrho[M_i]\cap U_{k_i}=\varnothing$. Then set $U_k:=\bigcap_{k<i\leq n} U_{k_i}\cap V_k$. 

    We can now define a map 
		\begin{align*}
			g_C&: \varrho[f^{-1}(C)]\to C,\\
			z&\mapsto x_i\iff z\in U_i.
		\end{align*}
    Since $C$ is a cluster, $g_C$ preserves $R^+$. Further, it is continuous because each $U_i$ is clopen.  Having defined $g_C$ for all clusters $C\subseteq F$, we can then define the desired $g:\greatest\fragment \spa X \to \spa F$ by setting 
        \[
		g(\varrho(z)):=\begin{cases}
			f(z)&\text{ if } f(z)\text{ does not belong to any proper cluster, }\\
			g_C(\varrho(z))&\text{ if }f(z)\in C\text{ for some proper cluster }C\subseteq F.
		\end{cases}
		\]
    Since both $f$ and each $g_C$ are continuous and preserve $R^+$, it follows that $g$ is a pre-stable map. 

    We now show that $g$ satisfies the BFC for $\spa{D}$. Let $\spa{d}\in \spa{D}$. For the ``back'' part, let $x\in X$ and suppose there is $y\in \spa{d}$ with $Rg(\varrho(x))y$. By construction, $g(\varrho(x))$ belongs to the same proper or improper $R^+$-cluster as $f(x)$, so we also have $R f(x) y$. Since $f$ satisfies the BFC for $\spa{D}$, there must be $z\in X$ with $Rxz$ and $f(z)\in \spa{d}$. By \Cref{seemaximals}, wlog, we may assume that $z\in \mathit{max}(f^{-1}(f(z)))$. But by construction this implies $g(\varrho(z))=f(z)$. Since $\varrho$ preserves $R$, we have $R \varrho (x)\varrho (z)$, and we have found our desired witness. 

    For the ``forth'' part, let $\spa{d}\in \spa{D}$, let $x\in X$ and suppose there is $y\in X$ with $g(\varrho (y))\in \spa{d}$ and $R\varrho(x)\varrho(y)$. Since $\varrho$ reflects $R$, we also have $Rxy$. Moreover, $g(\varrho (y))$ and $f(y)$ belong to the same cluster, whence $\upset[R]\cap f^{-1}(\spa{d})\neq\varnothing$. Since $f$ satisfies the BFC for $\spa{D}$, there must be some $z\in \spa{d}$ such that $f(x) z$. But now $f(x)$ and $g(\varrho (x))$ also belong to the same cluster, whence $Rg(\varrho (x))z$. We have thus shown that $g$ satisfies the BFC for $\spa{D}$.
\end{proof}

\begin{theorem}[Skeletal generation theorem]
    Every universal class of $\logic{K4.Grz}$-algebras is generated by its skeletal elements. That is, $\class{U}=\greatest \fragment \class{U}$ holds for every universal class $\class{U}$ of  $\logic{K4.Grz}$-algebras.\label{skellygen}
\end{theorem}
\begin{proof}
    By \Cref{mainlemma}, $\op{ThR}(\class{U})=\op{ThR}(\greatest \fragment \class U)$, so $\class{U}=\greatest \fragment \class{U}$ follows using \Cref{algebraization}.
\end{proof}

This was the challenging part of the argument for our main results. The rest of the argument consists of routine reasoning, which we repeat here for completeness. To begin with, we can now prove that the syntactic and semantic versions of the maps $\greatest, \least$ and $\fragment$ correspond to one another, in the following sense. 
\begin{lemma}
        Let $\logic{L}\in\lat{NExt}(\logic{mHC})$ and $\logic{M}\in \lat{NExt}(\logic{K4})$. Then  \label{algcommute}
        \begin{enumerate}
            \item $\op{Alg}(\greatest \logic L)=\greatest \op{Alg}(\logic L)$; \label{algcommutegreatest}
            \item $\op{Alg}(\least \logic L)=\least \op{Alg}(\logic L)$;\label{algcommuteleast}
            \item $\op{Alg}(\fragment \logic M)=\fragment \op{Alg}(\logic M)$;\label{algcommutefragment}
        \end{enumerate}
    \end{lemma}
    \begin{proof}
        (\ref{algcommutegreatest}) In view of \Cref{skellygen}, it suffices to show that $\op{Alg}(\least \logic L)$ and $\least \op{Alg}(\logic L)$ have the same skeletal elements. So let $\alg M$ be a skeletal $\logic{K4.Grz}$ algebra. Assume $\alg M\in \greatest\op{Alg}(\logic L)$. Since $\greatest\op{Alg}(\logic L)$ is generated by $\{\greatest\alg H: \alg H\in \op{Alg}(\logic L)\}$ as a universal class, by \Cref{rhosigmaid} and \Cref{translationrho} it follows that $\alg M\models T(\Gamma/\Delta)$ for every $\Gamma/\Delta\in \logic L$. This implies $\alg M\in \op{Alg}(\greatest\logic L)$. Conversely, assume $\alg M\in \op{Alg}(\greatest \logic L)$. Then $\alg M\models T(\Gamma/\Delta)$ for every $\Gamma/\Delta\in \logic L$. By \Cref{translationrho}, this is equivalent to $\fragment \alg M\in \op{Alg}(\logic L)$, so $\alg M=\greatest\fragment \alg M\in \greatest\op{Alg}(\logic L)$.

        (\ref{algcommuteleast}) If $\alg M$ is a $\logic{K4}$-algebra, then $\alg M\in \op{Alg}(\least \logic L)$ iff $\alg M\models T(\Gamma/\Delta)$ for all $\Gamma/\Delta\in \logic L$ iff $\fragment \alg M\models \Gamma/\Delta$ for all $\Gamma/\Delta\in \logic L$ iff $\fragment \alg M\in \op{Alg}(\logic L)$ iff $\alg M\in \fragment \op{Alg}(\logic L)$.

        (\ref{algcommutefragment}) Let $\alg H$ be a frontal Heyting algebra. If $\alg H\in \fragment \op{Alg}(\logic M)$, then $\alg H=\fragment \alg M$ for some $\alg M\in \op{Alg}(\logic M)$. It follows that $\alg M\models T(\Gamma/\Delta)$ whenever $T(\Gamma/\Delta)\in \logic M$, and by  \Cref{translationrho} in turn $\alg H\models \Gamma/\Delta$. So, $\alg H\in \op{Alg}(\fragment \logic M)$. Conversely, if $\fragment \op{Alg}(\logic M)\models \Gamma/\Delta$, then by \Cref{translationrho} $\op{Alg}(\logic M)\models T(\Gamma/\Delta)$, hence $\Gamma/\Delta\in \fragment \logic M$. This implies $\op{Alg}(\fragment \logic M)\subseteq\fragment \op{Alg}(\logic M)$.
    \end{proof}

    The result just established leads to a purely semantic characterization of mono\-modal companions. 
    \begin{lemma}
        Let $\logic{L}\in \lat{NExt}(\logic{mHC})$ and $\logic{M}\in \lat{NExt}(\logic{K4})$. Then $\logic{M}$ is a monomodal companion of $\logic{L}$ iff $\op{Alg}(\logic{L})=\fragment \op{Alg}(\logic{M})$.\label{semanticcomp}
    \end{lemma}
    \begin{proof}
        If $\logic L=\fragment \logic{M}$, then $\op{Alg}(\logic L)=\fragment \op{Alg}(\logic M)$ by \Cref{algcommute}. Conversely, assume $\op{Alg}(\logic L)=\fragment \op{Alg}(\logic M)$. By \Cref{sigmarhosub}, if $\alg H\in \op{Alg}(\logic L)$, then $\greatest \alg H\in \op{Alg}(\logic M)$. So, $\Gamma/\Delta\in \logic L$ iff $T(\Gamma/\Delta)\in \logic M$. 
    \end{proof}

    Now to the main results of this section. First, we prove that the monomodal companions of any $\sim$-rule system form an interval. 
    \begin{theorem}[Interval theorem]
        Let $\logic{L}\in \lat{NExt}(\logic{mHC})$.  The monomodal companions of $\logic{L}$ form an interval in $\lat{NExt}(\logic{K4})$, where the least and greatest companions are given by $\least\logic{L}$ and $\greatest\logic{L}$. \label{intervaltheorem}
    \end{theorem}
    \begin{proof}
        By \Cref{algcommute}, it suffices to show that  $\logic{M}$ is a monomodal companion of $\logic{L}$ iff $\greatest\op{Alg}(\logic L)\subseteq\op{Alg}(\logic M)\subseteq \least \op{Alg}(\logic L)$. Assume $\logic{M}$ is a monomodal companion of $\logic{L}$. By \Cref{semanticcomp} we have $\op{Alg}(\logic L)=\fragment \op{Alg}(\logic M)$, so clearly $\op{Alg}(M)\subseteq \least \op{Alg}(\logic L)$. To see that $\greatest \op{Alg}(\logic L)\subseteq\op{Alg}(\logic M)$, it suffices to show that every skeletal element of $\greatest \op{Alg}(\logic L)$ belongs to $\op{Alg}(\logic L)$. Let $\alg M\in \greatest \op{Alg}(\logic L)$ be skeletal. then $\fragment \alg M\in \op{Alg}(\logic L)$ by \Cref{translationrho}. So, there must be $\alg N\in \op{Alg}(\logic M)$ such that $\fragment \alg M=\fragment \alg N$. This implies $\greatest\fragment \alg N=\greatest \fragment \alg M=M$. By \Cref{sigmarhosub}, we conclude $\alg M\in \op{Alg}(\logic M)$.

        Conversely, assume $\greatest\op{Alg}(\logic L)\subseteq\op{Alg}(\logic M)\subseteq \least \op{Alg}(\logic L)$. By \Cref{sigmarhosub}, it follows that $\fragment \greatest\op{Alg}(\logic{L})=\op{Alg}(\logic L)$, so $\fragment\op{Alg}(\logic M)\supseteq\op{Alg}(\logic L)$. But by the definitions of $\fragment, \least$ we have $\fragment\op{Alg}(\logic M)=\fragment \least\op{Alg}(\logic L)$, so also $\fragment\op{Alg}(\logic M)\subseteq\op{Alg}(\logic L)$. By \Cref{semanticcomp}, it follows that $\logic{M}$ is a monomodal companion of $\logic{L}$.
    \end{proof}

    Second, we prove the following analog of the Blok-Esakia theorem, which yields the Kuznetsov-Muravitsky isomorphism as a special case. It was announced by Esakia \cite{Esakia2006TMHCaCMEotIL}. We refer to it as an \emph{Esakia theorem}. 
    
    \begin{theorem}[Esakia theorem for $\sim$-rule systems]
        The map $\greatest$ and the restriction of $\fragment$ to $\lat{NExt}(\logic{K4.Grz})$ are mutually inverse complete lattice isomorphisms between $\lat{NExt}(\logic{mHC})$ and $\lat{NExt}(\logic{K4.Grz})$. \label{blokesakia}
    \end{theorem}
    \begin{proof}
        It suffices to show that the semantic maps $\greatest:\lat{Uni}(\op{fHA})\to \lat{Uni}(\op{K4.Grz})$ and $\fragment:\lat{Uni}(\op{K4.Grz})\to \lat{Uni}(\op{fHA})$ are complete lattice isomorphisms and mutual inverses. Both maps are evidently order-preserving, and preservation of infinite joins follows from \Cref{translationrho}. Let $\class{U}\in \lat{Uni}(\op{K4.Grz})$. Then $\class{U}=\greatest\fragment U$ by \Cref{skellygen}, so $\greatest$ is surjective and a left inverse of $\fragment$. If $\class V\in \lat{Uni}(\op{fHA})$, then $\fragment \greatest\class U=\class U$ by \Cref{rhosigmaid}, so $\fragment$ is surjective and a left inverse of $\sigma$. Thus, $\greatest$ and $\fragment$ are mutual inverses, and therefore must both be bijections.  
    \end{proof}
    \begin{corollary}[Kuznetsov-Muravitsky isomorphism for $\sim$-rule systems]
         The restriction of  $\greatest$ to $\lat{NExt}(\logic{KM})$ and the restriction of $\fragment$ to $\lat{NExt}(\logic{GL})$ are mutually inverse complete lattice isomorphisms between $\lat{NExt}(\logic{KM})$ and $\lat{NExt}(\logic{GL})$.\label{kuzmura}
    \end{corollary}
    \begin{proof}
        The corollary follows from the observation that  $\greatest\logic{KM}= \logic{GL}$. We know that $\greatest\logic{KM}\subseteq \logic{GL}$ from \Cref{sigmataucoincide}. For the other direction, it is enough to observe that every skeletal $\logic{GL}$-space is of the form $\greatest{X}$ for some $\logic{KM}$-space $\spa{X}$. 
    \end{proof}

    We note that both \Cref{blokesakia} and \Cref{kuzmura} remain true when restricted to lattices of \emph{logics} only. 
     \begin{corollary}[Esakia theorem for $\sim$-logics]
        The restriction of the mappings  $\greatest$ and $\fragment$ to $\lat{NExtL}(\logic{mHC})$ and $\lat{NExtL}(\logic{K4.Grz})$ respectively are mutually inverse complete lattice isomorphisms between $\lat{NExtL}(\logic{mHC})$ and $\lat{NExtL}(\logic{K4.Grz})$. \label{blokesakialogics}
    \end{corollary}
    \begin{proof}
        By construction, $\greatest$ and $\fragment$ preserve the property of being a logic. Given \Cref{blokesakia}, the restrictions of $\greatest$ and $\fragment$ to logics are mutual inverses, so it suffices to show that they commute with meets and joins. For joins, this is obvious, as the join of two logics is a logic. For meets, we show that $\greatest$ commutes with $\op{Taut}$. For if this condition holds, we may reason as follows:
        \begin{align*}
            \greatest (\logic L\otimes_{\lat{ExtL}(\logic{mHC})} \logic L')&=\greatest(\op{Taut}(\logic{L}\otimes_{\lat{Ext}(\logic{mHC})}\logic L')) &\text{by \Cref{logicmeet}}\\
            &=\op{Taut}(\greatest(\logic{L}\otimes_{\lat{Ext}(\logic{mHC})}\logic L'))\\
            &=\op{Taut}(\greatest \logic L\otimes_{\lat{Ext}(\logic{mHC})}\greatest \logic L')&\text{by \Cref{blokesakia}}\\
            &=\greatest \logic L\otimes_{\lat{ExtL}(\logic{mHC})}\greatest \logic L'.&\text{by \Cref{logicmeet}}.
        \end{align*}
        Indeed, since $\greatest$ is order-preserving and  $\op{Taut}(\logic L)\subseteq \logic L$ we have $\greatest \op{Taut}(\logic L)\subseteq \greatest\logic L$. Since $\op{Taut}$ is also order-preserving and $\op{Taut}(\greatest \op{Taut}(\logic L))=\greatest \op{Taut}(\logic L)$, also $\greatest\op{Taut}(\logic L)\subseteq \op{Taut}(\greatest\logic L)$. Likewise, $\fragment\op{Taut}(\logic M)\subseteq \op{Taut}(\fragment\logic M)$ for every $\logic M\in \lat{NExt}(\logic{K4})$. Together, the last two claims imply $\greatest\logic L\subseteq\greatest \op{Taut}(\logic L)$.
      \end{proof}
     \begin{corollary}[Kuznetsov-Muravitsky isomorphism]
         The restriction of  $\greatest$ to $\lat{NExtL}(\logic{KM})$ and the restriction of $\fragment$ to $\lat{NExtL}(\logic{GL})$ are mutually inverse complete lattice isomorphisms between $\lat{NExtL}(\logic{KM})$ and $\lat{NExtL}(\logic{GL})$.
    \end{corollary}
    \begin{proof}
        Follows from \Cref{blokesakialogics} the same way \Cref{kuzmura} follows from \Cref{blokesakia}.
    \end{proof}
    
\section{Translations of Pre-Stable Canonical Rules}
\label{sec:trans}

In this last section, we characterize the translation $T$ in terms of pre-stable canonical rules. We then apply our characterization to establish a few more results: the announced strengthening of \Cref{admitsfiltrationclm} concerning pre-filtrations of Magari algebras, and two preservation theorems concerning the mapping $\sigma$. For the remainder of the paper, we focus only on $\sim$-rule systems above $\logic{KM}$, since our axiomatizations in terms of $\sim$ pre-stable canonical rules do not extend below $\logic{KM}$.

\subsection{Classicization and the Rule Translation Lemma}

The main difference between $\sim$ and $\clm$ pre-stable canonical rules is that the former have two domains, while the latter have one. We can get by with having a single domain in the $\clm$ case because $\square^+$ is a compound operator, defined in terms of $\square$. When constructing a pre-filtration of a $\clm$-rule $\Gamma/\Delta$, if $\square^+\varphi\in \mathit{Sfor}(\Gamma/\Delta)$, then also $\square\varphi\in \mathit{Sfor}(\Gamma/\Delta)$. So, the way we make sure that $\square^+\varphi$ receives the same valuation in the filtration as in the original model is simply by making sure that $\square \varphi$ does. 

The translation $T$ associates $\to$ with $\square^+$ and $\boxtimes$ with $\square$. However, $\to$ is not defined in terms of $\boxtimes$. We thus need to keep track of when $\to$ and $\boxtimes$ need to be fully preserved separately; syntax is of no help. This generates a problem: how do we turn the two domains in  $\scrsi{F}{\spa D}$ to turn the latter into a $\clm$ pre-stable canonical rule, given that there are no obvious connections between the two domains?

The solution to this problem is to restrict attention to a particular kind of $\sim$ pre-stable canonical rules, in which $\spa D^\leq$ has a natural embedding into $\spa D^\sqsubset$. The correspondence ensures that when checking the refutability of such rules, we need only verify that the witnessing pre-stable surjection satisfies the BFC$^\sqsubset$ for $\spa D^\sqsubset$.
\begin{definition}
    A $\sim$ pre-stable canonical rule $\scrsi{H}{D}$ is \emph{classicizable} when $(a, b)\in  D^\to$ implies $a\to b\in  D^\boxtimes$.
\end{definition}
Dually, $\scrsi{F}{\spa D}$ is classicizable when $\beta (a)\cap -\beta (b)\in \spa D^\to$ implies $\downset[\leq]{(\beta (a)\cap -\beta(b))}\in \spa D^\sqsubset$.

\begin{lemma}
    Let $\spa X, \spa Y$ be $\logic{GL}$-spaces and let $f:\spa X\to \spa Y$ be a pre-stable map satisfying the BFC$^R$ for a domain of the form $\spa E:=\{\downset[R^+]{\spa d}: \spa d\in \spa D\}$, with $\spa D$ any domain on $\spa Y$. Then $f$ also satisfies the BFC$^R$ for  $\spa D$. Moreover, the claim remains true if we let $\spa X, \spa Y$ be $\logic {KM}$-spaces and substitute $\sqsubset$ for $R$ and $\leq$ for $R^+$. \label{bfcdownset}
\end{lemma}
\begin{proof}
    Take $\spa d\in \spa D$ and $x\in X$. For the ``back'' condition, suppose there is $y\in \spa d$ with $Rf(x)y$. Then $y\in\downset[R^+]{\spa d}$.  Since $f$ satisfies the BFC for $\spa E$, there must be $z\in f^{-1}(\downset[R^+]{\spa d})$ with $Rxz$. By the properties of $\logic{GL}$-spaces, we may assume $z$ is maximal in $f^{-1}(\downset[R^+]{\spa d})$. By \Cref{maxmax} we have
    \begin{align*}
        \max{f^{-1}(\downset[R^+]{\spa d})}&=f^{-1}(\max{\downset[R^+]{\spa d}})\\
        &=f^{-1}(\max{{\spa d}})\\
        &\subseteq f^{-1}(\spa d),
    \end{align*}
    and we are done. 

    For the ``forth'' condition, suppose there is $y\in f^{-1}(\spa d)$ with $Rxy$. Then $y\in{f^{-1}(\downset[R^+]{\spa d})}$. Since $f$ satisfies the BFC for $\spa E$, there must be some $z\in \downset[R^+]{\spa d}$ such that $Rf(x) z$. By the properties of $\logic{GL}$-spaces, we may assume that $z\in \max{\downset[R^+]{\spa d}}=\max{\spa d}\subseteq \spa d$, and we are done. 
    The argument is completely analogous in the case of $\logic{KM}$-spaces. 
\end{proof}
\begin{lemma}
    Let $\scrsi{F}{\spa D}$ be a classicizable $\sim$ pre-stable canonical rule. A $\logic{KM}$-space $\spa X$ refutes $\scrsi{F}{\spa D}$ iff there is a pre-stable surjection $f:\spa X\to \spa F$ satisfying the BFC$^\sqsubset$ for $\spa D^\sqsubset$.
\end{lemma}
\begin{proof}
    The left-to-right direction is obvious and the right-to-left direction follows from \Cref{bfcdownset}.
\end{proof}

We can always restrict attention to classicizable rules without loss of generality. 
\begin{theorem}
    Every $\sim$-rule $\Gamma/\Delta$ is equivalent, over $\logic{KM}$, to finitely many classicizable $\sim$ pre-stable canonical rules. 
\end{theorem}
\begin{proof}
    The argument is essentially the same as that given in \Cref{rewrite}, with the following caveat. At the step where we filtrate a countermodel of $\Gamma/\Delta$ to construct a pre-stable canonical rule, using the construction from \Cref{admitsprefiltrationsim}, start out by defining \[D^\boxtimes:=\{V(\varphi):\boxtimes \varphi\in \Theta\}\cup\{V(\varphi\to \psi):\varphi\to \psi\in \mathit{Sfor}(\Gamma/\Delta)\}.\]
    Then run the rest of the construction exactly the same way. The rule obtained at the end is evidently classicizable. 
\end{proof}

We can now characterize the translation $T$. Let $\scrsi{F}{\spa D}$ be a classicizable $\sim$ pre-stable canonical rule. We define the \emph{classicization} $\ruletrans{F}{\spa D}$ of $\scrsi{F}{\spa D}$ by setting 
\[\ruletrans{F}{\spa D}:=\scrmod{\sigma F}{ D_\circ},\]
where $\spa D_\circ:=\spa D^\sqsubset$. 

We call a $\clm$ pre-stable canonical rule $\scr{F}{D}$ \emph{classicized} when $\spa F$ is a skeletal $\logic{GL}$ space and every  $\spa d\in \spa{D}$ is a downset. It should be clear that $\scr{F}{D}$ is classicized precisely when it is the classicization of some $\sim$ pre-stable canonical rule.

\begin{lemma}[Rule translation lemma]
    Let $\scrsi{F}{\spa D}$ be a classicizable $\sim$ pre-stable canonical rule and let $\spa X$ be a $\logic{GL}$-space. Then $\spa X\models \ruletrans{F}{\spa D}$ iff $\spa X\models T(\scrsi{ F}{\spa D})$.\label{ruletrans}
\end{lemma}
\begin{proof}
    We give only a proof sketch; more details can be found in \cite{BezhanishviliBETVSCR}, where a similar result is proved in a different signature. 
      $(\Rightarrow)$ Suppose $\spa X\not\models  T(\scrsi{F}{\spa D})$. By \Cref{translationrho} this implies $\fragment \spa X\not\models \scrsi{F}{\spa D}$. So, there is a pre-stable surjection $f:\fragment \spa X\to \spa F$ satisfying the BFC for $\spa D^{\leq}$ and $\spa D^{\sqsubset}$. Composing $f$ and the cluster collapse map $\varrho: \spa X\to \fragment  \spa X$ yields a pre-stable surjection $f \circ \varrho:\spa X\to \greatest\spa F$ that the BFC for $\spa D$. 

    $(\Leftarrow)$ Suppose $\spa X\not\models \ruletrans{F}{\spa D}$. Then there is a pre-stable surjection $f:\spa X\to \greatest \spa F$ that satisfies  the BFC for $\spa D$. We define a map $g: \fragment \spa X\to \spa F$ by setting $g(\varrho (x)):=f(x)$. This is well defined: since $\greatest\spa F$ is skeletal, elements of $\spa X$ that belong to the same cluster have the same image under $f$. Clearly, $g$ is pre-stable, surjective, and satisfies  the BFC$^\sqsubset$ for $\spa D^{\sqsubset}$. By \Cref{bfcdownset}, it also satisfies the BFC$^\leq$ for $\spa D^\leq$. 
\end{proof}

\subsection{Application: pre-filtrations of Magari algebras}

We now establish the strengthening of \Cref{admitsfiltrationclm} announced earlier. We will build pre-filtrations of models based on Magari algebras by taking a detour through frontons. Given a valuation $V$ with the right shape on a Magari algebra $\alg M$, we will first extract a finite fronton from $\fragment\alg M$, then define a pre-filtration of $(\alg M, V)$ based on $\greatest\fragment \alg M$.

\begin{lemma}
    Let $\alg{M, N}$ be Magari algebras. Let $E\subseteq N$ be such that if $a\in E$, then  $a=\bigwedge\{\neg a_i\lor a_j: (i, j)\in I_a\}$ for some finite $I_a$ with $a_i, a_j\in O^+(\alg M)$ for every $(i, j)\in I_a$. Let 
    \[D:=\bigcup_{a\in E}\{\neg a_i\lor a_j: (i, j)\in I_a\}.\]
    If a pre-stable embedding $h:\alg N\to \alg M$ satisfies the BDC$^\square$ for $D$, then it satisfies the BDC$^\square$ for $E$ as well. \label{bdcsplit}
\end{lemma}
\begin{proof}
        Let $a\in E$ and reason as follows:
        \begin{align*}
            h(\square a)&=h(\square \bigwedge \{\neg a_i\lor a_j: (i, j)\in I_a\}\\
            &=h(\bigwedge \{\square(\neg a_i\lor a_j): (i, j)\in I_a\}\\
            &=\bigwedge \{h(\square(\neg a_i\lor a_j)): (i, j)\in I_a\}\\
            &=\bigwedge \{\square h(\neg a_i\lor a_j): (i, j)\in I_a\}\\
            &=\square \bigwedge \{h(\neg a_i\lor a_j): (i, j)\in I_a\}\\
            &=\square h(a).
        \end{align*}
\end{proof}

\begin{theorem}
    Let $\alg{M}$ be a Magari algebra. If $\alg{M}\not\models \Gamma/\Delta$, then there is a model $(\alg{M}, V)$ such that refutes $\Gamma/\Delta$ and has a pre-filtration $(\alg{N}, V')$ through $\mathit{Sfor}(\Gamma/\Delta)$ based on a Magari algebra $\alg{N}$. \label{admitsfiltrationmagari}
\end{theorem}
\begin{proof}
Assume $\alg M\not\models \Gamma/\Delta$. Then, by \Cref{mainlemma}, also $\greatest\fragment \alg M\not\models \Gamma/\Delta$. Let $W$ be a valuation on $\greatest\fragment \alg M$ witnessing this fact. Define a valuation $V$ on $\alg M$ by setting $V(p):=\varrho^{-1}(W(p))$ for all $p\in \mathit{Prop}$ (recall that, by duality,  $\varrho^{-1}$ coincides with the canonical embedding of $\greatest \fragment \alg M$ into $\alg M$.)

We claim that every $a\in V[\mathit{Sfor}(\Gamma/\Delta)]$ is of the form 
\[a=\bigwedge_{(i, j)\in I_a} \neg a_i\lor a_j, \]
with each $I_a$ finite and $a_i, a_j\in O^+(\alg M)$ for each $(i, j)\in I_a$. This follows from the fact that $\greatest\fragment \alg M$ is skeletal and every element of $V[\mathit{Sfor}(\Gamma/\Delta)]$ is of the form $\varrho^{-1}(b)$ for some $b\in \greatest\fragment \alg M$. We fix one such decomposition for every $a\in V[\mathit{Sfor}(\Gamma/\Delta)]$.

Now, define:
\begin{align*}
    B_a&:=\{a_i, a_j :(i, j)\in I_a\},& C_a&:=\{\neg a_i\lor a_j :(i, j)\in I_a\}, \\
    B&:=\{B_a:a=V(\varphi), \varphi\in\mathit{Sfor}(\Gamma/\Delta) \},& C&:=\{C_a:a=V(\varphi), \varphi\in \mathit{Sfor}(\Gamma/\Delta) \},  \\
    S_a^\odot&:=\bigcup \{\odot (\neg a_i\lor a_j): (i, j)\in I_a\}&{} &\text{for $\odot\in \{\square^+, \square\}$},\\
    S^\odot&:=\bigcup \{S_a^\odot: a=V(\varphi), \text{ and } {\odot \varphi}, \varphi\in \mathit{Sfor}(\Gamma/\Delta)\}&{} &\text{for $\odot\in \{\square^+, \square\}$}.
\end{align*}
Think of each $a$ as constructed from quasi-opens in three steps. We start from the quasi-open ``building blocks'' of $a$. These are the elements of $B_a$. Then, the index set $I_a$ gives instructions on how to combine the quasi-open building blocks into into \emph{cells}, where the \emph{cells} are elements of the form $\neg a_i\lor a_j$ with $(i, j)\in I_a$. These are the elements of $C_a$. Finally, the cells are combined via conjunction to obtain $a$. With this picture in mind, each set $S^\odot_a$ contains all the $\odot$-necessitations of the cells of $a$. We put in $S^\odot$ precisely the $\odot$-necessitated cells of those elements $a$ with $\odot a\in V[\mathit{Sfor}(\Gamma/\Delta)]$. 

Let $A:=B\cup {S^{\square^+}}\cup S^{\square}$ and for $a\in V[\mathit{Sfor}(\Gamma/\Delta)]$ define
\begin{align*}
    D_a^\to&:= \{(a_i, a_j):(i, j) \in I_a\},\\
    D_a^\boxtimes&:=\{\square^+ (\neg a_i\lor a_j) : (i, j)\in I_a\},\\
    D^\to&:=\bigcup \{D^\to_a : a=V(\varphi)\text{ and }\square \varphi, \varphi\in \mathit{Sfor}(\Gamma/\Delta)\},\\
    D^\boxtimes&:=\bigcup \{D^\boxtimes_a : a=V(\varphi)\text{ and } \square \varphi, \varphi\in \mathit{Sfor}(\Gamma/\Delta)\}.
\end{align*}
Requiring $\square\varphi, \varphi\in \mathit{Sfor}(\Gamma/\Delta)$ in the definition of $D^\to$, instead of $\square^+\varphi, \varphi$, ensures that $\square^+(\neg a_i\lor a_j)\in D^\boxtimes$ iff $(a_i, a_j)\in D^\to$; this will be important later. 

We use the construction from the proof of \Cref{admitsprefiltrationsim} to construct a finite fronton $\alg K$ with the following properties:
\begin{enumerate}
    \item $\alg K$ is a bounded sublattice of $\fragment \alg M$ generated by a finite superset of $A$;
    \item The inclusion embedding $\subseteq : \alg K\to \fragment \alg M$ satisfies the BDC for $D:=(D^\to, D^\boxtimes)$.
\end{enumerate}
To do this, we enumerate $D^\boxtimes:=b_1, \ldots, b_n$ and construct a finite sequence $(\alg K_0, \ldots, $ $\alg K_n)$ of bounded sublattices of $\fragment \alg M$. Here $\alg K_0$ is the bounded sublattice of $\fragment \alg M$ generated by $A$ and $\alg K_{i+1}$ is obtained from $\alg K_i$ by adding complements in the Boolean sublattice $[b_{i+1}, \boxtimes b_{i+1}]$ to all elements of $\alg K_i$, then generating a bounded sublattice of $\fragment \alg M$ from the result. 

The inclusion embedding witnesses $\fragment \alg M\not\models \scrsi{K}{D}$. By choice of $D$, moreover, $ \scrsi{K}{D}$ is classicizable. So, by \Cref{ruletrans}, $\greatest\fragment \alg M\not\models \ruletrans{K}{D}$. Let $h:\greatest\alg K\to \greatest\fragment M$ be a pre-stable embedding witnessing this fact.  Since $\greatest\fragment \alg M$ is a subalgebra of $\alg M$, this yields a pre-stable embedding $k: \greatest\alg K\to \alg  M$ that satisfies the BDC for $D_\circ$. 


Define a valuation $V'$  on $\greatest\alg K$ by putting $V'(p)=k^{-1}(V(p))$ when $p\in \mathit{Sfor}(\Gamma/\Delta)$, and arbitrary otherwise.  Note that $\greatest\alg K$ is generated by a finite superset of  the set $V'[\mathit{Sfor}(\Gamma/\Delta)]$ as a Boolean algebra. Consider  the  domain $E:=\{V(\varphi):=\square\varphi\in \mathit{Sfor}(\Gamma/\Delta)\}.$ 
Note $a\in E$ implies $a=\bigwedge U$ for some $U\subseteq D$. So, by \Cref{bfcdownset}, $k$ satisfies the BDC for $E$. We have thus established that the model $(\alg K, V')$ is a pre-filtration of $(\alg M, V)$ through $\mathit{Sfor}(\Gamma/\Delta)$.
\end{proof}

The result just established allows us to axiomatize every $\clm$-rule system above $\logic{GL}$ entirely in terms of \emph{classicized} pre-stable canonical rules. 
\begin{theorem}
    Every $\clm$-rule $\Gamma/\Delta$ is equivalent, over $\logic{GL}$, to finitely many classicized $\clm$ pre-stable canonical rules.\label{rewriteclassicized}
\end{theorem}
\begin{proof}
   Consider all the pre-filtrations $(\alg K, V')$ of countermodels $(\alg M, V)$ of $\Gamma/\Delta$ that can be obtained via the construction used in the previous theorem, identified up to isomorphism.  Let $\Phi$ be the corresponding set of $\clm$ pre-stable canonical rules $\ruletrans{\alg K}{D}$, with $D:=(D^\boxtimes, D^\to)$. By choice of $D^\boxtimes, D^\to$, the rule $\scrsi{\alg K}{D}$ is classicizable, hence $\ruletrans{\alg K}{D}$ is well defined. 

   By the properties of pre-filtration, a Magari algebra $\alg M$ refutes $\Gamma/\Delta$ iff it refutes some $\ruletrans{\alg K}{D}\in \Phi$. By compactness, then, there must be a finite subset $\Psi\subseteq \Phi$ such that a Magari algebra $\alg M$ refutes $\Gamma/\Delta$ iff it refutes some $\ruletrans{\alg K}{D}\in \Psi$.
\end{proof}

Note that the proof just given is less constructive than that of \Cref{rewrite}, in that we are not able to explicitly give an upper bound to the size of $\Psi$. This is because, in the proof of \Cref{admitsfiltrationmagari}, the size of the set generating $\alg K$ depends not only on $\mathit{Sfor}(\Gamma/\Delta)$, but also on properties of the specific countermodel we are pre-filtrating.

\subsection{Application: Preservation Results}

In this last section, we prove that the mapping $\sigma$, above $\logic{KM}$, preserves Kripke completeness and the finite model property. 
\begin{lemma}
    Let $\spa X$ be a Kripke frame for a $\sim$-logic $\logic L\in \op{NExt}(\logic{KM})$. Then $\greatest\spa X$ is a Kripke frame for $\greatest \logic L$.\label{sigmakripke}
\end{lemma}
\begin{proof}
    By \Cref{translationrho}, $\greatest\spa X$ validates every rule $T(\Gamma/\Delta)$ for $\Gamma/\Delta\in \logic L$, and by \Cref{maxirrkripke} it is also a $\logic{GL}$-frame. Given \Cref{sigmataucoincide}, this suffices to conclude that $\greatest\spa X$ is a Kripke frame for $\greatest \logic L$.
\end{proof}
\begin{theorem}[Cf. {\cite[Prop. 23]{Muravitsky2014LKaB}}]
    Let $\logic{L}\in \op{NExt}(\logic{KM})$. 
    \begin{enumerate}
        \item $\logic{L}$ is Kripke complete iff  $\tau \logic L$ is Kripke complete.
        \item $\logic{L}$ has the finite model property iff   $\tau \logic L$ has the finite model property.
    \end{enumerate}
\end{theorem}
\begin{proof}
    Recall that $\least$ and $\greatest$ coincide above $\logic{KM}$ and $\logic{GL}$ (\Cref{sigmataucoincide}). So, it suffices to prove the version of the theorem obtained by substituting $\greatest$ for $\least$. 
    
    \emph{Kripke completeness}. $(\Leftarrow)$ Assume $\greatest \logic M$ is Kripke complete. Suppose $\logic L\not\vdash \Gamma/\Delta$. Then $\greatest\logic L\not\vdash T(\Gamma/\Delta)$. So, there is a Kripke frame $\spa X$ for $\greatest\logic L$ such that $\spa X\not\models T(\Gamma/\Delta)$. By \Cref{translationrho}, we have $\fragment \spa X\not\models \Gamma/\Delta$. Since, by \Cref{translationrho} once more, $\fragment \spa X$ is a Kripke frame for $\logic L$, we are done. 
    
    $(\Rightarrow)$ It suffices to show that for every $\clm$ pre-stable canonical rule $\scrmod{F}{\spa D}$, if $\greatest \logic L\not\vdash \scrmod{F}{\spa D}$, then there is a Kripke frame for $\greatest \logic L$ that refutes $\scrmod{F}{\spa D}$. So suppose $\greatest \logic L\not\vdash \scrmod{F}{\spa D}$. Then there is a modal space $\spa X$ for $\greatest \logic L$ that refutes $\scrmod{F}{\spa D}$. By \Cref{rewriteclassicized}, there is a finite set $\Phi$ of classicized $\clm$ pre-stable canonical rules whose conjunction is equivalent to $\scrmod{F}{\spa D}$ over $\logic{GL}$. Since $ \spa X$ is a $\logic{GL}$-space, there must be some $\ruletrans{G}{\spa E}\in \Phi$ with $ \spa X\not\models \ruletrans{G}{\spa E}$. By  \Cref{ruletrans} and \Cref{translationrho}, it follows that $\fragment \spa X\in \op{Spa}(\logic L)$ refutes $\scrsi{G}{\spa E}$. Since $\logic L$ is Kripke complete, there is a Kripke frame $\spa Y$ for $\logic L$ that refutes $\scrsi{G}{\spa E}$. Then $\sigma\spa Y$ is a Kripke frame for $\sigma L$, by \Cref{sigmakripke}. By \Cref{ruletrans} and \Cref{translationrho} once more, we have $\greatest\spa Y\not\models\ruletrans{G}{\spa E}$, which in turn implies $\scrsi{F}{\spa D}$.

    For the finite model property, we reason exactly as we just did, noting that if a $\clm$ Kripke frame $\spa X$ is finite, then so is $\fragment \spa X$, and 
    a $\sim$ Kripke frame $\spa{Y}$ is finite only if $\greatest\spa Y$ is. 
\end{proof}

\section{Conclusion}
In this paper we developed the technique of pre-stable canonical rules for the Kuznetsov–Muravitsky system $\logic{KM}$ and classical modal logics above $\logic {K4}$. We applied these rules to obtain  alternative proofs of the Kuznetsov–Muravitsky isomorphism and the Esakia theorem,  as well as of some preservation results. 

There are several  possible directions for further research. One would be to investigate the impact that the technique of pre-stable canonical rules could have on the logic $\logic{KM}$. One could try to develop a theory of pre-stable logics for $\logic{KM}$, in analogy with the theory of stable logics from \cite{BezhanishviliEtAl2016SCR, BezhanishviliEtAl2016CSL}. Similarly, it would be interesting to obtain concrete axiomatizations of extensions of $\logic{KM}$ via pre-stable canonical rules.

It would also be natural to explore the potential impact of the method of pre-stable canonical rules on other intuitionistic modal logics. The technique of stable canonical rules for intuitionistic modal logics has already been applied by \cite{Liao2023SCRfIML} and \cite{melzer2020canonical}. However, pre-stable canonical rules might be applicable in cases where stable canonical rules are not—for example, in certain logics that resemble $\logic{KM}$. We leave this as an open-ended question for future work.

Finally, it be would interesting to explore whether algebra-based rules can be developed for all extensions of $\logic{mHC}$, as opposed to only extensions of $\logic{KM}$. This would allow us to extend the preservation results from \Cref{sec:trans} to extensions of $\logic{mHC}$.

\vspace{3mm}
\noindent
{\bf Acknowledgement}\  The authors are very grateful to the editors of this volume for their patience and for their valuable comments and pointers to the literature. They also wish to thank Rodrigo Almeida for many helpful discussions on topics related to this paper. The first-named author would like to acknowledge the support of the MSCA-RISE (Marie Skłodowska-Curie Research and Innovation Staff Exchange) project MOSAIC, grant agreement No.~101007627, funded by Horizon 2020 of the European Union.

\bibliographystyle{default}
\bibliography{database.bib}

\end{document}